\documentclass[11pt,final]{amsart}
\usepackage{a4wide}

\usepackage{amsmath,amsthm,amssymb, amsthm}
\usepackage[style=alphabetic, backend=bibtex]{biblatex}
\bibliography{cm-intersection}
\usepackage{url}

\usepackage{enumitem}
\usepackage{float}
\usepackage[all]{xy}
\usepackage[rgb,dvipsnames]{xcolor}
\usepackage[latin1]{inputenc}
\usepackage[pdfborder={0 0 0}, colorlinks = true, linkcolor=blue, citecolor=OliveGreen, pdftex]{hyperref}
\usepackage[capitalize]{cleveref}

\newcommand{\R}{\mathbb{R}} 
\newcommand{\Q}{\mathbb{Q}} 
\newcommand{\uhp}{\mathbb{H}} 
 
\newcommand{\B}{\mathbb{B}} 
\newcommand{\C}{\mathbb{C}} 
 
\newcommand{\Z}{\mathbb{Z}} 
\newcommand{\F}{\mathbb{F}} 
 
\newcommand{\N}{\textup{N}}

\newcommand{\Zhat}{\hat{\mathbb{Z}}}
\newcommand{\adeles}{\mathbb{A}}

\newcommand*{\domain}{\mathbb{D}}

\newcommand{\OD}{\calO_{D}}
\newcommand{\ODt}{\calO_{D}^{\times}}

\newcommand{\ODhatt}{\hat{\calO}_{D}^{\times}}

\newcommand{\sage}{\texttt{sage}}

\newcommand{\bs}{\backslash}

\newcommand{\calC}{\mathcal{C}}

\newcommand{\calG}{\mathcal{G}}

\newcommand{\calM}{\mathcal{M}}

\newcommand{\calO}{\mathcal{O}}

\newcommand{\calQ}{\mathcal{Q}}

\newcommand{\calX}{\mathcal{X}}
\newcommand{\calY}{\mathcal{Y}}
\newcommand{\calZ}{\mathcal{Z}}

\newcommand{\different}[1]{\mathfrak{d}_#1}

\newcommand{\fraka}{\mathfrak{a}}
\newcommand{\frakb}{\mathfrak{b}}
\newcommand{\frakc}{\mathfrak{c}}

\newcommand{\frakf}{\mathfrak{f}}

\newcommand{\frakn}{\mathfrak{n}}
\newcommand{\frakp}{\mathfrak{p}}
\newcommand{\frakP}{\mathfrak{P}}
\newcommand{\frakQ}{\mathfrak{Q}}

\newcommand{\legendre}[2]{\left( \frac{\mathstrut #1}{#2} \right)}

\newcommand{\abcd}{\begin{pmatrix}a & b \\ c & d\end{pmatrix}}

\newcommand{\abs}[1]{\left\vert#1\right\vert}

\DeclareMathOperator{\End}{End}
\DeclareMathOperator{\Lie}{Lie}
\DeclareMathOperator{\Aut}{Aut}

\DeclareMathOperator{\SL}{SL}
\DeclareMathOperator{\GL}{GL}

\DeclareMathOperator{\Mp}{Mp}
\DeclareMathOperator{\Cl}{Cl}
\DeclareMathOperator{\Clk}{Cl_{k}}

\DeclareMathOperator{\tr}{tr}

\DeclareMathOperator{\Gal}{Gal}
\DeclareMathOperator{\divisor}{div}
\DeclareMathOperator{\Og}{O}
\DeclareMathOperator{\SO}{SO}

\DeclareMathOperator{\GSpin}{GSpin}

\DeclareMathOperator{\ord}{ord}

\DeclareMathOperator{\Sch}{Sch}

\DeclareMathOperator{\pr}{pr}

\DeclareMathOperator{\cRes}{cRes}

\newtheorem{theorem}{Theorem}[section]
\newtheorem{proposition}[theorem]{Proposition}
\newtheorem{lemma}[theorem]{Lemma}
\newtheorem{corollary}[theorem]{Corollary}

\newtheorem{theoremintro}{Theorem}
\newtheorem{propositionintro}[theoremintro]{Proposition}

\theoremstyle{definition}
\newtheorem{definition}[theorem]{Definition}

\newtheorem{remark}[theorem]{Remark}
\newtheorem*{remark*}{Remark}

\newcommand{\QD}[1][]{Q_\Delta\ifthenelse{\equal{#1}{}}{}{\left(#1\right)}}

\renewcommand{\j}{\mathbf{j}}
\newcommand{\jD}{\mathbf{j}_D}
\newcommand{\jDz}{\mathbf{j}_{D_0}}

\newcommand{\zmatrix}{\begin{pmatrix}  z & -z^{2} \\ 1 & -z \end{pmatrix}}

\DeclareMathOperator{\Diff}{Diff}

\DeclareMathOperator{\Spec}{Spec}

\DeclareMathOperator{\Pic}{Pic}

\newcommand{\deghat}{\widehat{\deg}}

\usepackage{stmaryrd}

\usepackage{mathrsfs}

\numberwithin{equation}{section}

\hypersetup
{
    pdfauthor={Stephan Ehlen},
    pdftitle={Singular moduli of higher level and special cycles},
    pdfkeywords={CM values, singular moduli, modular curves}
}

\numberwithin{equation}{section}
\begin{document}
\title{Singular moduli of higher level and special cycles}
\author{Stephan Ehlen}
\email{stephan.ehlen@mcgill.ca}
\address{McGill University, Department of Mathematics and Statistics, 805 Sherbrooke St. West, Montreal, Quebec, Canada H3A 0B9}
\thanks{This work is partially supported by DFG grant BR-2163/4-1.}
\subjclass[2010]{11F11, 11G15}

\date{\today}

\begin{abstract}
  We describe the complex multiplication (CM) values of modular functions for $\Gamma_0(N)$
  whose divisor is given by a linear combination of Heegner divisors 
  in terms of special cycles on the stack of CM elliptic curves.
  In particular, our results apply to Borcherds products of weight $0$ for $\Gamma_0(N)$.
  By working out explicit formulas for the special cycles, we obtain the prime ideal factorizations of such CM values.
\end{abstract}
\maketitle

\section{Introduction}

Let $d < 0$ be a negative fundamental discriminant and denote by $\calQ_{d}$
the set of (positive and negative definite) integral binary quadratic forms of discriminant $d$. 
For every $Q  \in \calQ_{d}$ given by $Q(x,y) = ax^{2} + bxy + cy^{2}$
the unique root of $Q(\tau,1)=0$ in $\uhp$, the complex upper half-plane, is denoted $\alpha_{Q}$.
These points are called CM points because the associated elliptic curve  has complex multiplication (CM). 
That is, its endomorphism ring is an order in an imaginary quadratic field (in our case the ring of integers). 
The values of $j(\tau)$ of the $j$-invariant at CM points are classically called singular moduli
\cite{weber-algebra-3, zagiertraces}. 
For $Q \in \calQ_d$, the value $j(\alpha_{Q})$ is an algebraic integer which
generates the Hilbert class field $H$ of $k_{d} = \Q(\sqrt{d})$.

In their seminal article, Gross and Zagier \cite{grosszagier-singularmoduli} consider the modular function of level one
\begin{equation*}
  \Psi(z,d) = \prod_{Q \in \SL_2(\Z) \bs \calQ_{d}} (j(z)-j(\alpha_{Q}))^{1/w_{d}},
\end{equation*}
for $z \in \uhp$, and where $w_{d}$ is the number of roots of unity in $k_{d}$.
They proved that for a negative fundamental discriminant $D$ coprime to $d$,
we have
\begin{equation}
  \label{eq:gz}
  \prod_{Q \in \SL_2(\Z) \bs \calQ_{D}}\Psi(\alpha_{Q},d)^{2/{w_{D}}}
  = \pm \prod_{\substack{x\in \Z,\, n,n'>0\\ 4nn'=dD -x^2}}n^{\epsilon(n')},
\end{equation}
where $\epsilon(n') = \pm 1$ is given explicitly in \cite{grosszagier-singularmoduli}.

Gross and Zagier give an analytic and an algebraic proof of \eqref{eq:gz}.
The latter is given for prime discriminants only, but does in fact provide 
the prime ideal valuations of the values $\Psi(\alpha_Q,d)$. 
Dorman \cite{dorman-j} generalized this result to fundamental discriminants.
Moreover, the contributions at the finite places to the height pairing studied
in \cite{gkz} essentially provide a generalization to higher level.

In this paper, we refine results of Bruinier and Yang \cite{bryfaltings} to relate the prime
factorization of special values of certain modular functions to special cycles on
the stack of CM elliptic curves. This has been used in \cite{ehlen-diss, ehlen-cmtheta}
to relate the coefficients of harmonic weak Maa\ss{} forms of weight one to these special cycles
and to study the modularity of related arithmetic generating series.
We use these results to determine the prime ideal factorization of such modular functions explicitly.
In \cref{sec:borcherds-products} we show that we can apply our results to Borcherds products of weight zero.

We now define the higher level singular moduli we are considering.
Let $d$ be a negative integer congruent to a square modulo $4N$ and let
$\mathcal{Q}_{d,N,r}$ be the set of integral binary quadratic forms 
$\left[a,b,c\right]=ax^2+bxy+cy^2$ of discriminant $d=b^2-4ac$ 
such that $a$ is congruent to $0$ modulo $N$ and $b$ is congruent to $r$ modulo $2N$.
For each form $Q = \left[a,b,c\right] \in \mathcal{Q}_{d,N,r}$
there is an associated CM point $\alpha_Q=\frac{-b+\sqrt{d}}{2a}$ 
in $\uhp$. The group $\Gamma_0(N)$ acts on $\mathcal{Q}_{d,N,r}$ with finitely many orbits.
We define the Heegner divisor $Z(d,r)$ to be the $\Gamma_0(N)$-invariant divisor (with rational coefficients)
\[
   Z(d,r) = \sum_{Q \in \calQ_{d,N,r}} \frac{\alpha_Q}{w_Q},
\]
where $w_Q$ is the order of the stabilizer of $Q$ in $\Gamma_0(N)$.
We also denote the image of $Z(d,r)$ in $Y_0(N) = \Gamma_0(N) \bs \uhp$ and
its compactification $X_0(N)$ in the same way.

Now fix a square-free negative fundamental discriminant $D$ and let $H$ be the Hilbert class field of the imaginary quadratic field $k = k_{D}$
of discriminant $D$. We will write $\calO \subset k$ for the ring of integers in $k$ and $\N(\cdot) = \N_{k/\Q}(\cdot)$ for the norm of $k/\Q$.
The ring of integers in $H$ is denoted $\calO_H$.
We fix an embedding of $H$ into $\C$ and for $\rho \in \Z$ with $\rho^2 \equiv D \bmod{4N}$ 
the integral ideal
\[
  \frakn = \left(N,\frac{\rho+\sqrt{D}}{2}\right)
\]
of norm $N$ in $k$.
Consider the Heegner point $z_{D,\rho}$ contained in $Z(D,\rho)$ given by
\[
  z_{D,\rho} = \frac{\rho+\sqrt{D}}{2N} \in \uhp.
\]
Representatives for the remaining points in $Z(D,\rho)$ can be obtained as follows.
If $\fraka$ is an integral ideal of $k$, then $\fraka\frakn$ is an ideal of norm divisible by $N$.
If we let $\alpha,\beta$ be generators of $\fraka\frakn$, such that $\alpha = \N(\fraka)N$, then the point $\beta/\alpha$ lies in $\uhp$ and as the class $[\fraka]$ runs through the ideal class group of $k$, these points run through half of a system of representatives of $Z(D,\rho)$. The other half is obtained by also considering $-\bar{z}$ for each of these points.

We will now give a simpler version of our main result, \cref{thm:divf2}, for prime level.
Let $N=p$ be a prime and consider a meromorphic modular function $f$ on $Y_0(p)$.
We assume in the introduction that $f$ is invariant under the Fricke involution $z \mapsto -1/pz$ and has a divisor on $\uhp$ of the form
  \[
    \divisor(f) = \sum_{r \bmod{2p}} \sum_{\substack{d\, \in\, \Z_{<0} \\ d\, \equiv\, r^2 \bmod{4p}}} c(d,r) Z(d,r).
  \]
Moreover, we assume that $\divisor(f)$ and $Z(D, \rho)$ have no common points
and that the Fourier coefficients of $f$ are coprime integers.
In particular, we have that $f \in \Q(j(\tau), j(p\tau))$.
By the theory of complex multiplication, the value $f(z_{D,\rho})$ is contained in the Hilbert class field $H$.
The following theorem gives the multiplicity of a prime of the Hilbert class field
in the prime ideal factorization of $f(z_{D,\rho})$ in terms of certain special cycles $\calZ(m, \frakn, \mu)$ on $\Spec \calO_{H}$ that parametrize elliptic curves with special endomorphisms of norm $m\N(\frakn)$ satisfying a congruence condition given by $\mu \in \different{k}^{-1}\frakn/\frakn$ (see \hyperref[sec:moduli-cm-elliptic]{Sections}~\ref{sec:moduli-cm-elliptic} and~\ref{sec:spec-endom-p} for definitions). Here, $\different{k}^{-1}$ denotes the inverse different of $k$.
We write $\calZ(m,\frakn,\mu)$ as a formal sum
\[
 \calZ(m, \frakn, \mu) = \sum_{\frakP} \calZ(m, \frakn, \mu)_\frakP \frakP,
\]
where the sum runs over all nonzero prime ideals of $H$ and $\calZ(m, \frakn, \mu)_\frakP \in \Q$.
These cycles are generalizations of those first considered
by Kudla, Rapoport and Yang \cite{kry-tiny} in connection with the coefficients 
of the derivative of an incoherent Eisenstein series of weight one.
\begin{theoremintro}
  \label{thm:divf2-intro}
  We have for every prime ideal $\frakP$ of the Hilbert class field
  \[
    \ord_{\mathfrak{P}}(f(z_{D,\rho})) = w_{D} \sum_{r \bmod{2p}} \sum_{d\, \in\, \Q_{<0}} c(d,r)
        \sum_{\substack{n\, \equiv\, \rho \cdot r \bmod{2p} \\ n^{2} \leq dD}}
        \calZ\left(\frac{dD - n^{2}}{4p\abs{D}}, \frakn, \frac{n+r\sqrt{D}}{2\sqrt{D}} \right)_{\frakP}.
  \]
\end{theoremintro}

\begin{remark*}\ 
  \begin{enumerate}
  \item The set of CM values of $f$ at points in $Z(D,\rho)$ forms a single Galois orbit and the valuations of other points in $Z(D,\rho)$
can be obtained using Shimura reciprocity and \cref{prop:Zprime-intro} below.
\item Note that there exist modular functions whose divisor is given by a linear combination of Heegner divisors
on $Y_0(N)$ that are not obtained as Borcherds products.
This has been remarked by Borcherds \cite[Example 5.2]{bogkz} and follows
from Theorem 7.7 of \cite{bryfaltings}. We also remark that Bruinier \cite{jan-converse} proved a converse
theorem for Borcherds products which shows that the case of modular curves
is quite exceptional in this regard.
  \end{enumerate}
\end{remark*}

To obtain formulas for the quantities in \cref{thm:divf2} or \cref{thm:divf2-intro} above, we use Gross' formula 
for the length of the local rings of $\calZ(m, \frakn, \mu)$ as in \cite{kryderivfaltings}, where the arithmetic degrees of these cycles have been computed.
In the case of a prime discriminant these formulas are completely explicit and rather simple, as we shall see below.

Fix a fractional ideal $\frakn \subset k$ and let $\mu \in \different{k}^{-1}\frakn/\frakn$ and $m \in \Q_{>0}$.
We write $[\fraka]$ for the class of the fractional ideal $\fraka$ in the class group $\Cl_k$ of $k$.
We define a set of rational primes by
\begin{equation*}
  \label{eq:Diff-intro}
  \Diff(m) = \{ p < \infty\ \mid\ (-m\N(\fraka),D)_p = -1 \}
\end{equation*}
and for $n \in \Z$ and $[\frakb] \in \Cl_k$ we let
\[
  \rho(n,[\frakb]) = \# \{ \fraka \subset \calO \ \mid\ \N(\fraka) = n,\, \fraka \in [\frakb] \}.
\]
If $m$ is a rational number and $p$ is a rational prime which is non-split in $k$, we define
\begin{equation*}
  \nu_{p}(m) =
  \begin{cases}
    \frac{1}{2}(\ord_{p}(m)+1), & \text{if $p$ is inert in $k$},\\
    (\ord_{p}(m \abs{D})), & \text{if $p$ is ramified in $k$}.
  \end{cases}
\end{equation*}

We obtain the following result (see \cref{prop:Zprime}).
\begin{propositionintro}\ 
  \label{prop:Zprime-intro}
  Let $D=-l$ for a prime $l \equiv 3 \bmod{4}$.
  \begin{enumerate}
  \item We have $\calZ(m, \frakn, \mu)_{\frakP} = 0$ unless $\abs{\Diff(m)} = 1$ and $m + \N(\mu)/\N(\fraka) \in \Z$.
  \item Assume that $\Diff(m) = \{p\}$ and that $m + \N(\mu)/\N(\fraka) \in \Z$. 
    There is a unique prime ideal $\frakP_0 \mid p$ fixed by complex conjugation, $\bar\frakP_{0} = \frakP_{0}$.
    For $\frakP = \frakP_{0}^{\sigma}$, where $\sigma = \sigma(\frakb)$
    corresponds to the ideal class of $\frakb$ under the Artin map $(\cdot,H/k)$, we have
\[
    \calZ(m,\fraka,\mu)_{\frakP} =
      2^{o(m)-1} \nu_{p}(m) \rho(m\abs{D}/p, [\fraka\frakb^{-2}]),
  \]
where $o(m) = 1$ if $\ord_{l}(ml)>0$ and $o(m) = 0$, otherwise.
  \end{enumerate}
\end{propositionintro}

\subsection{Acknowledgements}
This article contains some results that were contained in my thesis \cite{ehlen-diss}.
First and foremost I would like to thank my advisor Jan Hendrik Bruinier for his constant support and encouragement.
I also thank Claudia Alfes and Tonghai Yang for their help. I also thank Benjamin Howard for helpful discussions and in particular for helping me with the proof of \cref{prop:compActions}.
Moreover, the helpful comments of the anonymous referee are greatly appreciated.

\subsection{Outline}
The paper is organized as follows: In \cref{sec:prelims}, we set up the notation and give basic definitions.
The moduli stack of CM elliptic curves is introduced in \cref{sec:moduli-cm-elliptic} and in \cref{sec:spec-endom-p} we describe the special cycles and
prove the explicit formulas for the multiplicities.
The integral model for the modular curve $X_0(N)$ will be recalled in \cref{sec:an-integral-model} and in \cref{sec:cm-modular-functions}
we prove the main result of this paper (\cref{thm:divf2}).
The applications to Borcherds products are discussed in \cref{sec:borcherds-products} and finally,
 \cref{sec:examples} covers some numerical examples illustrating our results.

\providecommand{\zsmatrix}[1]{\bigl ( \begin{smallmatrix} #1 \end{smallmatrix}  \bigr )}

\section{Preliminaries and notation}
\label{sec:prelims}
We write $\Z$ for the integers, $\Q$ for the field of rational numbers,
$\Z_p$ for the ring of $p$-adic integers and $\Q_p$ for the field of $p$-adic numbers. 
We also use the notation $\Zhat = \prod_p\Z_p$, where the product is over all primes.
We write $\adeles$ for the adeles of $\Q$ and use the notation $\adeles_f$ for the finite adeles of $\Q$ 
If $k$ is a number field, we write $\calO_k$ for its ring of integers and 
$\adeles_k$ (and $\adeles_{k,f}$) for the adeles (and finite adeles) of $k$. 
Moreover, we write $\hat\calO_k$ for the tensor product $\calO_k \otimes_\Z \Zhat$.

\subsection{Modular forms}
\label{sec:scal-valu-modul}
We briefly recall the definition of modular forms for $\Gamma_0(N)$.
We refer to one of the standard references
for details, for instance \cite{koblitz, miyake-mf, onowebmod, wstein-modforms}.

Let $N$ be a positive integer and consider the congruence
subgroup $\Gamma_0(N) \subset \SL_2(\Z)$ defined by
\[
\Gamma_0(N) = \left\lbrace \abcd \in \SL_2(\Z)\ \mid\ c \equiv 0 \bmod{N} \right\rbrace.
\]
For an integer $k \in \Z$ and a matrix $\gamma \in \GL_{2}^{+}(\Q)$,
the \emph{Petersson slash operator} is defined on functions
$f: \uhp \rightarrow \C$ by
\[
(f\mid_{k}\gamma)(\tau) = (c\tau+d)^{-k} \det(\gamma)^{k/2} f(\gamma \tau), \text{ where } \gamma = \abcd \in \GL_{2}^{+}(\Q).
\]
\begin{definition}
  A meromorphic function $f: \uhp \rightarrow \C$ is called
  a \emph{meromorphic modular form} of weight $k \in \Z$ for $\Gamma_0(N) \subset \SL_2(\Z)$ if
  \begin{enumerate}
  \item $(f\mid_{k}\gamma)(\tau)= f(\tau)$ for all $\gamma \in \Gamma_0(N)$
  \item and $f$ is meromorphic at the cusps of $\Gamma_0(N)$.
  \end{enumerate}
\end{definition}
The conditions at the cusps can be phrased in terms of Fourier expansions.
A meromorphic modular form $f$ has a Fourier expansion (at the cusp $\infty$) of the form
\[
f(\tau) = \sum_{n \gg -\infty}^{\infty}a_{f}(n)q^{n},
\]
where $q = e^{2 \pi i \tau} = e(\tau)$.
That is, if $f$ is a meromorphic modular form, then there are only finitely many non-vanishing Fourier coefficients
$a_f(n)$ with $n<0$.
The conditions at the other cusps are similar.
By a modular function for $\Gamma_0(N)$, we mean a meromorphic modular form of weight $0$ for $\Gamma_0(N)$.
It is well known \cite{shimauto} that the field of modular functions for $\Gamma_0(N)$ is generated by $j(\tau)$ and $j_N(\tau) = j(N\tau)$.

\subsection{Elliptic curves}
\label{sec:elliptic-curves}
We recall the definition of an elliptic curve over an arbitrary base scheme. 
We refer to Chapter 2 of the book by Katz and Mazur \cite{katz-mazur}
and the two books by Silverman \cite{silverman-ec, silverman-advanced} for details.

\begin{definition}
  \label{def:ec}
  Let $S$ be a scheme. A proper smooth curve $E \rightarrow S$
  with geometrically connected fibers of genus one together
  with a section $0: S \rightarrow E$, is called an \emph{elliptic curve}.
\end{definition}
It is an important fact that an elliptic curve admits a unique structure
as a group scheme \cite[Theorem 2.1.2]{katz-mazur}.
It is well known (see \cite[Corollary 9.4]{silverman-ec} and \cite{deuring-cm}) that
the endomorphism ring of an elliptic curve over a field
is either $\Z$, an order in an imaginary quadratic field
or an order in a quaternion algebra.

We will use the notation
\[
\B = \legendre{a,b}{F},
\]
to denote the quaternion algebra $F \oplus F\alpha \oplus F\beta \oplus F\alpha\beta$
with $\alpha^2 = a$ and $\beta^{2} = b$ and $\alpha\beta = -\beta\alpha$.
We say that the quaternion algebra $\B$ is \emph{split} over $F$
if $\B$ is isomorphic to $M_2(F)$. Here, $M_2(F)$ is the algebra of $2 \times 2$-matrices over $F$.

For a quaternion algebra $\B$ over $\Q$, we say that $\B$ is split at a prime $p$ (or $\infty$),
if $\B \otimes_{\Q} \Q_p$ is split over $\Q_p$. Otherwise, we say that
$\B$ is \emph{ramified} at $p$ (or $\infty$). We use the same convention over other fields.

Let $E/k$ be an elliptic curve over a field $k$.
If the endomorphism ring $\End_k(E)$ contains an
order $\calO \subset k$ of an imaginary quadratic field $k$,
then we say that $E$ has \emph{complex multiplication} by $\calO$ (or just that $E$ has CM).

A \emph{CM elliptic curve} over a scheme $S$ is a pair $(E,\iota)$ consisting
of an elliptic curve $E/S$ and an action $\iota: \calO \hookrightarrow \End_S(E)$
of an order $\calO$ in an imaginary quadratic field on $\End_S(E)$.

If $\End_S(E) = \calO \subset \B$ is an order in a quaternion algebra,
then we say that $E$ is \emph{supersingular}.
Supersingular elliptic curves only occur in positive characteristic
(cf. VI., Theorem 6.1 of \cite{silverman-ec}).

\subsection{The Hilbert class field}
\label{sec:hilbertclass}
We recall only very briefly some facts from class field theory
that we need later, in particular the Artin map.
We refer to \cite[Chapter II]{silverman-advanced} or \cite{shimauto} for details.
Let $k$ be a totally imaginary field and let $L$
be a finite abelian extension of $k$. Write $\calO_{L}$ for the ring of integers in $L$.
Let $\frakp$ be a prime ideal of $k$ that does not ramify in $L$.
There is a unique element $\sigma_{\frakp} \in \Gal(L/k)$, such that
\[
\sigma_{\frakp}(x) \equiv x^{\N_{L/k}(\frakp)} \bmod{\frakP}
\]
for all $x \in \calO_{L}$ and any prime ideal $\frakP \mid \frakp$.
If $\frakc$ is an integral ideal of $k$ which is divisible by all primes that ramify in $L/k$
and $I(\frakc)$ is the group of fractional ideals that are relatively prime to $\frakc$,
then the Artin map is defined as
\[
(\cdot, L/k): I(\frakc) \rightarrow \Gal(L/k),
\]
by extending the map $\sigma_{\frakp}$ linearly.

We let $k$ be an imaginary quadratic field and $H$ be the Hilbert class field
of $k$. This is the maximal unramified abelian extension of $k$. It corresponds to the ray class field for $\frakc=1$.
By class field theory, we obtain in this case \cite[II, Example 3.3]{silverman-advanced}
an isomorphism via the Artin map
\[
(\cdot, H/k): \Clk \rightarrow \Gal(H/k).
\]
We will use the convention that we write $\sigma(\fraka) = \sigma([\fraka])$
for $([\fraka],H/k)$, the image of the (class of the) fractional ideal $\fraka$ under this map.

\section{Moduli of CM elliptic curves}
\label{sec:moduli-cm-elliptic}
In this section we recall some facts about the moduli stack of CM elliptic curves
and special $0$-cycles on it. These cycles will play an important role
in the description of the values of modular functions on $Y_{0}(N)$.

Let $S$ be a scheme. By a scheme over $S$ or an $S$-scheme, we mean a scheme
$X$ together with a morphism $\pi: X \rightarrow S$.
We denote by $(\Sch/S)$ the category of schemes over $S$.
The morphisms in this category are morphisms of $S$-schemes, that is,
morphisms that are compatible with the given morphisms to $S$.
Here and throughout, if $X$ is a stack over a base scheme $S$ and $T$ is an $S$-scheme,
we abbreviate
\[
  X_{/T} = X \times_{S} T.
\]
If $T = \Spec R$ with $R$ a ring, we simply write $X_{/R}$ for $X_{/\Spec R}$.

Here and for the rest of this section, let $D$ be
a negative fundamental discriminant and denote by $k = k_{D}$ the imaginary quadratic field
of discriminant $D < 0$. We write $\OD$ for the ring of integers of $k$
and let $H = H_D$ be the Hilbert class field of $k$.

We consider the moduli problem
which assigns to a base scheme $S$ over $\OD$ the category $C_D^{+}(S)$ of pairs $(E,\iota)$, where
\begin{enumerate}
\item $E$ is an elliptic curve over $S$ with
complex multiplication $\iota: \OD \hookrightarrow \End(E)$,
\item such that the induced map
  \begin{equation}
    \label{eq:lie}
    \Lie(\iota): \OD \rightarrow \End_{\calO_{S}}(\Lie E) = \calO_{S},
  \end{equation}
  coincides with the structure map $S \to \Spec(\OD)$.
\end{enumerate}
The morphisms in this category are isomorphisms respecting the actions.

Moreover, we denote by $C_{D} = \cRes_{\OD/\Z}(C_{D}^{+})$ the restriction
of coefficients of $C_{D}^{+}$ to $\Z$ (in the sense of Grothendieck).
That is, the structure map of $C_{D}$ is given by $C_{D}^{+} \rightarrow \Spec (\OD) \rightarrow \Spec(\Z)$.
This describes the moduli problem without the normalization \eqref{eq:lie}.

\begin{proposition}
  \label{prop:Cstack}
  The moduli problem $C_{D}^{+}$ is represented by an algebraic stack,
  also denoted by $C_{D}^{+}$,
  which is smooth of relative dimension 0 and proper over $\Spec \OD$.
  If $R$ is a discrete valuation ring
  with algebraically closed residue field $\F$, the reduction map
  \[
    C_{D}^{+}(R) \rightarrow C_{D}^{+}(\F)
  \]
  is surjective.
  Consequently, $C_{D}$ is also represented by an algebraic stack of relative dimension $0$
  over $\Spec \Z$, which is finite (and thus proper).
\end{proposition}
\begin{proof}
  This is a consequence of the canonical lifting theorem \cite{howard-barbados-notes, Lan}.
  Properness follows from the fact that all points of $C_{D}$
  in characteristic 0 have potentially good reduction and the valuative criterion of properness.
  See \cite[Section 5]{kry-tiny} or \cite{bruinier-howard-yang-unitary} for details.
\end{proof}

\begin{lemma}
  \label{lem:Ccoarse}
  The coarse moduli scheme $\mathbf{C}_{D}^{+}$ of $C_{D}^{+}$ is isomorphic to $\Spec \calO_{H}$
  as a scheme over $\OD$.
  Consequently, the coarse moduli scheme $\mathbf{C}_{D}$ of $C_D$ is isomorphic
  to $\Spec \calO_{H}$ as a scheme over $\Z$.
\end{lemma}
\begin{proof}
  This is \cite[Corollary 5.4]{kry-tiny}. There, the authors prove the corresponding
  isomorphism for $C_{D}^{+}$, whose coarse moduli scheme is $\Spec \calO_{H} \rightarrow \Spec \OD$.
\end{proof}

We denote by
\[
  \pr: C_{D}^{+} \rightarrow \mathbf{C}_{D}^{+}
\]
the canonical map to the coarse moduli scheme.
\begin{proposition}
  \label{prop:CdlocRings}
  Let $\xi \in C_{D}^{+}$ be a geometric point and let $\pr(\xi) = \bar\xi$
  be the corresponding point of $\mathbf{C}_{D}^{+}$. Then
  \[
    \hat\calO_{C_{D}^{+},\xi} = \hat\calO_{\mathbf{C}_{D}^{+},\bar\xi},
  \]
  where $\hat\calO_{C_{D}^{+},\xi}$ and $\hat\calO_{\mathbf{C}_{D}^{+},\bar\xi}$ denote the
  completions of the \'{e}tale local rings at $\xi$ and $\bar\xi$, respectively.
\end{proposition}
\begin{proof}
  This is Corollary 5.2 in \cite{kry-tiny}.
\end{proof}

We now describe the geometric points of $C_{D}$ in every characteristic.
The following construction is very important for us.

Recall that over $\C$, we have a canonical bijection
\begin{equation}
  \label{eq:CDC}
  C_D^{+}(\C) \cong k^{\times} \bs \adeles_{k,f}^{\times} / \ODhatt,
\end{equation}
given by the theory of complex multiplication \cite{silverman-advanced}.
Here and throughout, we write $\adeles_{k,f}^\times$ for the finite ideles over $k$.
To an idele $h \in \adeles_{k,f}^{\times}$ that corresponds to the ideal class $[(h)]$,
the bijection assigns the (isomorphism class of the) elliptic curve with complex points
\[
  E(\C) = \C/(h).
\]
Moreover, if $(E,\iota) \in C_{D}^{+}(\C)$ is given by
$(\C/\Lambda, \iota)$, then multiplication with
$h \in \adeles_{k,f}^{\times}$ on the right hand side of \eqref{eq:CDC}
corresponds to
\[
 E \mapsto (h) \otimes_{\OD} E,
\]
where $(h) \otimes_{\OD} E$ is the elliptic curve over $\C$ with
complex points
\[
 ((h) \otimes_{\OD} E)(\C) \cong \C/(h)\Lambda.
\]
This defines an action of the class group $\Clk$ on the set of
isomorphism classes of CM elliptic curves (with CM by $\OD$) over $\C$.

Now let $(E,\iota) \in C_D^{+}(S)$ for a scheme $S$ and let $h \in \adeles_{k,f}^{\times}$,
corresponding to the ideal $(h)$. Then we can define a functor from the category
of $S$-schemes to the category of $\OD$-modules by
\[
  T \mapsto (h) \otimes_{\OD} E(T).
\]
This functor is in fact represented by an elliptic curve over $S$
and the construction is called the \emph{Serre construction}.
We denote the elliptic curve representing this functor by
$h.E = (h) \otimes_{\OD} E$.
For details, the reader may consult \cite{howard-morningside},
\cite[Section 7]{conrad-gz}.

We follow the description given in \cite{kry-tiny}
to describe the geometric points of $C^{+}_{D}$ in positive characteristic.
\begin{proposition}[Corollary 5.5 of \cite{kry-tiny}]
  Let $\frakp$ be a prime ideal of $k$ and let $\overline{\kappa(\frakp)}$
  denote an algebraic closure of the residue field $\kappa(\frakp)$.
  We have a bijection
  \[
    C^{+}_{D}(\overline{\kappa(\frakp)}) \cong k^{\times} \bs \adeles_{k,f}^{\times} / \ODhatt.
  \]

  The action by the Frobenius automorphism over $\kappa(\frakp)$ on the left hand side
  corresponds to the translation by an idele of the form $(1,\ldots,1,\pi,1,\ldots)$,
  where $\pi$ is a uniformizer at $\frakp$.
\end{proposition}

The proposition establishes a simply transitive action of the class group $\Clk$
on the points $C_{D}^{+}(\bar\F)$ over any algebraically closed field $\bar\F$.
We also have a bijection on geometric points
$C_{D}^{+}(\bar\F) \cong \Spec \calO_{H}(\bar\F)$.
On the points $\Spec \calO_{H}(\bar\F)$, we have an
action of the Galois group $\Gal(H/k)$.

Fix a morphism $\pr: C_{D}^{+} \rightarrow \Spec \calO_{H}$.
Then $\pr$ induces an isomorphism
\[
    \pr_{\F}: C_{D}^{+}(\bar\F) \cong \Spec \calO_{H}(\bar\F)
\]
on geometric points over any algebraically closed field $\bar\F$.
For $\bar\F = \overline{k}$ or $\bar\F = \overline{\kappa(\frakp)}$, the group $\Gal(H/k)$ acts on both sides.
It acts naturally on the right hand side and the action on the left hand side is given
via the isomorphism
\[
    k^{\times} \bs \adeles_{k,f}^{\times} / \ODhatt \cong \Gal(H/k), h \mapsto \sigma((h))
\]
given by the Artin map $\sigma$ of class field theory and the action of the idele class group given above.

The next proposition shows the compatibility of these bijections and group actions.
\begin{proposition}
  \label{prop:compActions}
  With the notation as above,
  the coarse moduli space map $\pr$ is compatible with these actions.
  More precisely, we have
  \[
    \pr_{\F}h.(E,\iota) = (\pr_{\F}(E,\iota))^{\sigma^{-1}(h)}.
  \]
\end{proposition}
\begin{proof}
  The compatibility over $\C$ is contained in the main theorem of complex multiplication \cite[II, Theorem 8.2]{silverman-advanced}.
  (Note that our  normalization of the action of the class group is different from the one used by Silverman.)
  Precisely, we have for an elliptic curve $E = \C/\Lambda$ that $h.E = \C/(h)\Lambda$ and $j(h.E)=j^{\sigma(h^{-1})}(E)$.
  Therefore, if we assume without loss of generality that
  $\pr_{\C}(E,\iota) = \lambda : \calO_{H} \hookrightarrow \C$ is the embedding given by
  $j \mapsto j(E)$, then $\pr_{\C}(h.(E,\iota)) = j \mapsto j(h.E) = j^{\sigma^{-1}(h)}(E) = \lambda^{\sigma^{-1}(h)}$.
  
  We will use this to prove the statement over $\kappa(\frakp)$ for a fixed prime $\frakp$ of $k$.
  Fix an isomorphism $\C_{\frakp} \cong \C$, where $\C_\frakp$ is the completion of an algebraic closure
  of $k_{\frakp}$.
  Then we obtain isomorphisms
  \[
    C_D^{+}(\C) \cong C_D^{+}(\C_\frakp) \cong C_D^{+}(\overline{\kappa(\frakp)})
  \]
  by \cref{prop:Cstack}.
  All of these bijections are compatible with the Serre construction.
  Similarly, we have bijections
  \[
    \Spec \calO_H(\C) \cong \Spec \calO_{H}(\C_\frakp) \cong \Spec \calO_{H}(\overline{\kappa(\frakp)}).
  \]
  The key part is now that the diagram
  \begin{equation*}
    \xymatrix{
        C_D^{+}(\C) \ar[r] \ar[d] & C_D^{+}(\C_\frakp) \ar[r] \ar[d] & C_D^{+}(\overline{\kappa(\frakp)}) \ar[d] \\
        \Spec \calO_H(\C)\ar[r]  & \Spec \calO_{H}(\C_\frakp) \ar[r] & \Spec \calO_{H}(\overline{\kappa(\frakp)})
    }
  \end{equation*}
  is commutative and the bijection $C_D^{+}(\C) \cong \Spec\calO_{H}(\C)$
  is compatible with the actions, as stated above.
  The bijections in the lower row
  are compatible with the action of the Galois group.
  Consequently, the bijection
  $C_D^{+}(\bar\F_p) \rightarrow \Spec \calO_{H}(\overline{\kappa(\frakp)})$
  is compatible with the two actions, as well.
\end{proof}

\section{Special endomorphisms}
\label{sec:spec-endom-p}
For $(E,\iota) \in C_{D}(S)$ we write $\calO_{E} = \End_{S}(E)$
and consider the lattice $L(E,\iota)$ of \emph{special endomorphisms}
\[
  L(E,\iota) = \{x \in \calO_{E}\ \mid\ \iota(\alpha)x
             = x\iota(\bar\alpha) \text{ for all } \alpha \in \OD \text{ and } \tr{x} = 0 \}
\]
as in Definition 5.7 of \cite{kry-tiny}.
It is equipped with the positive definite quadratic form
$\N(x) := \deg(x) = -x^{2}$. For $S = \Spec \C$ or $S = \Spec \bar\F_{p}$ for a prime $p$ that is
split in $k$, we have that $L(E,\iota)$ is zero.

For a non-split prime $p$ and $S=\Spec(\bar\F_p)$), $L(E,\iota)$ is a positive definite lattice of
rank $2$ in $\calO_{E}$ and $(E,\iota)$ is supersingular.
In this case $\calO_{E}$ is a maximal order in the quaternion algebra $\B_p$
over $\Q$, which is ramified exactly at $p$ and $\infty$.

Fix a fractional ideal $\fraka \subset k$ and let
$\mu \in \different{k}^{-1}\fraka/\fraka$ and $m \in \Q_{>0}$.
We write $Q(\mu) = \N(\mu)/\N(\fraka)$, which is well defined as an element of $\Q/\Z$.

The following moduli problem has been studied in \cite{kry-tiny} and
\cite{bryfaltings} and generalized in \cite{ky-pullback}.
To a scheme $S$ we assign the category $\calZ(S)$ of triples $(E,\iota,x)$, where
\begin{enumerate}
  \item $(E,\iota) \in C_{D}(S)$,
  \item $x \in L(E,\iota)\different{k}^{-1}\fraka$, such that
    \[
      \N(x) = m \N (\fraka), \quad x + \mu \in \calO_{E}\fraka.
    \]
\end{enumerate}
Here, we also wrote $\N(x)$ for the reduced norm in $\B_p$.
If $\calZ(S)$ is non-empty, then we have $m + \N(\mu)/N(\fraka) \in \Z$.

\begin{lemma}[Lemma 6.2 in \cite{bryfaltings}]
  The moduli problem $\calZ$ is represented by an algebraic stack $\calZ(m,\fraka,\mu)$
  of dimension 0 and the forgetful map $\psi: \calZ(m,\fraka,\mu) \rightarrow C_{D}$ defined by
  $(E,\iota,x) \mapsto (E,\iota)$ is finite and \'{e}tale.
\end{lemma}

For $m \in \Q_{>0}$, we define a set of rational primes by
\begin{equation}
  \label{eq:Diff}
  \Diff(m) = \{ p < \infty\ \mid\ (-m\N(\fraka),D)_p = -1 \}.
\end{equation}

\begin{remark}
  By the product formula for the Hilbert symbol \cite[III.2, Theorem 3]{serrearith}, we have
  \[
    \prod_{p \leq \infty} (-m\N(\fraka),D)_{p} = 1.
  \]
  But since $(-m\N(\fraka),D)_{\infty} = -1$, the cardinality of $\Diff(m)$ is odd.
  Moreover, if $p \in \Diff(m)$, then $p$ is non-split.
\end{remark}

\begin{lemma}\ 
  \begin{enumerate}
  \item If $\abs{\Diff(m)} > 1$, then $\calZ(m, \fraka, \mu) = \emptyset$.
  \item If $\Diff(m) = \{p\}$, then $p$ is non-split in $k$
    and $\calZ(m, \fraka, \mu)(\bar\F_{q}) = \emptyset$ for $q \neq p$.
  \end{enumerate}
\end{lemma}
\begin{proof}
  If there is an element $(E,\iota,x) \in \calZ(m, \fraka, \mu)$, then
  this shows that we have an isomorphism of quaternion algebras
  \[
    \B_p \cong \left( \frac{D, -m \N(\fraka)}{\Q} \right).
  \]
  However, since $\B_p$ is ramified exactly at $p$ and $\infty$,
  this is equivalent to
  \[
    (D,-m\N(\fraka))_{v} =
    \begin{cases}
     -1, & \text{if }v = p, \infty,\\
     1, & \text{otherwise.}
   \end{cases}
  \]
  This condition is equivalent to $\Diff(m) = \{p\}$.
\end{proof}

In the notation of \cite[Section 3]{vistoli-intersection},
the stack $\calZ(m,\fraka,\mu)$ defines a $0$-cycle
$\psi_{*}[\calZ(m,\fraka,\mu)]$ on $C_{D}$ since the forgetful map is proper.

Moreover, note that the map $\pr: C_{D} \rightarrow \Spec \calO_{H}$ is also proper
by \cref{lem:Ccoarse}. Therefore, we can consider
the proper pushforward $\pr_{*}[\calZ(m,\fraka,\mu)]$ to $\Spec \calO_{H}$.
In our case
\[
  \pr_{*}[\calZ(m,\fraka,\mu)] = \frac{1}{w_k}[\pr(\calZ(m,\fraka,\mu))]
\]
because the automorphism group of a general geometric point of $C_{D}$ is $\ODt$ \cite{kry-tiny}.

By abuse of notation,
we also denote by $\calZ(m,\fraka,\mu)$ the corresponding divisor on
the coarse moduli scheme and simply write
\[
  \calZ(m,\fraka,\mu) = \sum_{\frakP \subset \calO_{H}} \calZ(m,\fraka,\mu)_{\frakP} \frakP.
\]
If there is no confusion possible, we simply write $\calZ(m)$ or $\calZ(m)_{\frakP}$.
Note that the multiplicities above are \emph{the same} for the pushforward from $C_{D}^{+}$ and from
$C_{D}$.

In what follows, we will find formulas for the multiplicities $\calZ(m,\fraka,\mu)_{\frakP}$.
From now on, fix a prime $p$ that is non-split in $k$ and assume that $m \equiv -Q(\mu) \bmod{\Z}$.
Let $p_{0} \in \Z$ be a prime with $p_{0} \nmid 2pD$ such that if $p$ is inert in $k$, we have
  \[
   (D,-pp_{0})_{v} =
   \begin{cases}
     -1, & \text{if }v=p,\infty,\\
     1, & \text{otherwise,}
   \end{cases}
  \]
and if $p$ is ramified in $k$, we have
  \[
    (D,-p_{0})_{v} =
    \begin{cases}
     -1, & \text{if }v=p,\infty,\\
     1, & \text{otherwise.}
   \end{cases}
  \]
With this choice, put
\[
  \kappa_{p} =
  \begin{cases}
    pp_{0}, & \text{ if } p \text{ is inert in } k,\\
    p_{0}, & \text{ if } p \text{ is ramified in } k,
  \end{cases}
\]
and let $\frakp_{0}$ be a fixed prime ideal of $\OD$ lying above $p_{0}$.
Here, $( \cdot, \cdot)_{v}$ denotes the $v$-adic Hilbert symbol.
The existence of such a prime $p_0$ follows essentially from Dirichlet's theorem, see also
\cite[III, Theorem 4]{serrearith}. Note that the genus of $[\frakp_{0}]$ is well defined
since the symbols $(D, \cdot )_{v}$ form a basis for the genus characters.

We can write $\B_p = k \oplus k y_{0}$, where $y_{0}^{2} = \kappa_{p}$,
similar to the situation in the last section.
Here, the decomposition is orthogonal with respect to the bilinear form
corresponding to the reduced norm of $\B_p$.
We write $[\gamma,\delta]$ for the element $\gamma + \delta y_0 \in \B_p$.

\begin{proposition}
  \label{prop:Lact}
  Let $p$ be a prime that is non-split in $k$ and let $(E,\iota) \in C_D(\bar\F_p)$.
  Then $L(E,\iota)$ is a projective $\OD$-module of rank 1 and
  there is a fractional ideal $\frakb \subset k$, such that
  \[
    L(E,\iota) \cong \frakb \bar\frakb^{-1} \frakp_0^{-1} y_{0}.
  \]
  Here $y_{0} \in \calO_{E}$ with $\N(y_{0}) = \kappa_{p}$.
  Moreover, if $h \in \adeles_{k,f}^{\times}$, then
  \[
    L (h.(E,\iota)) = (h) \overline{(h)}^{-1} L(E,\iota).
  \]
\end{proposition}
\begin{proof}
  The first statement is Proposition 5.13 in \cite{kry-tiny}.
  The second follows from the action of the ideles on maximal orders in the quaternion algebra $\B_p$, as described in detail in Section 5 of \cite{kry-tiny}.
\end{proof}

Recall that there are two actions of $\adeles_{k,f}^{\times}$ on ideals of $k$.
One is given by the multiplication by the ideal $(h)$ corresponding to the idele
$h$ and the other one is given by the action of $\adeles_{k,f}^{\times} \cong T(\adeles_{f})$,
where $T = \GSpin_{U}$ for the quadratic space $U = k$ with quadratic form
given by the norm on $k$. For details, we refer to Section 2.2 of \cite{ehlen-binary}. 
To avoid confusion, we will denote the action of $h$ as an element of $T(\adeles_{f})$ by $h.\fraka = (h)\bar{(h)}^{-1}\fraka$ for any fractional ideal $\fraka \subset k$.
\begin{proposition}
  \label{prop:actZsigma}
  Let $h \in \adeles_{k,f}^{\times}$ and write $\sigma = \sigma(h)$
  for the element of $\Gal(H/k)$ under the Artin map.
  Then we have
  \[
    \calZ(m,\fraka,\mu)_{\frakP^{\sigma}} = \calZ(m,h^{-1}.\fraka,h^{-1}.\mu)_{\frakP}.
  \]
\end{proposition}
\begin{proof}
  This follows from Propositions \ref{prop:compActions}
  and \ref{prop:Lact}.
\end{proof}

\begin{lemma}
  \label{lem:quatord}
  Let $(E_{0},\iota_{0}) \in C_D(\overline{\F}_{p})$
  such that $[L(E_{0},\iota_{0})]^{-1} = [\frakp_{0}]$, where $[L(E_{0},\iota_{0})]$ denotes the
  class of the rank one $\OD$-module in $\Pic(\OD)$.
  Then the maximal order $\End(E_{0})$ of $\B_p$ can be described in the following way.
  
  If $p$ is inert in $k$, let $\frakc_{0} = \frakp_{0}\different{k}$.
  If $p$ is ramified and $\frakp \subset \OD$ is the prime above $p$,
  let $\frakc_{0} = \frakp_{0}\frakp^{-1}\different{k}$.
  
  There exists a generator $\lambda_0$ of $\different{k}^{-1}\frakc_{0}/\frakc_{0}$
  with
  \[
     \N(\lambda_0) \equiv -\kappa_{p} \bmod{\N(\frakc_{0})},
  \]
  such that
  \[
    \End(E_{0},\iota_{0}) = \calO_{\frakc_{0},\lambda_0,\B_p},
  \]
  where
  \[
    \calO_{\frakc_{0},\lambda_0,\B_p}
    =  \{ [\gamma, \delta] \ \mid\ \gamma \in \different{k}^{-1},
       \ \delta \in \frakc_{0}^{-1},\ \gamma + \lambda \delta \in \OD\}
  \]
  is a maximal order in $\B_p$.

  Moreover, if $(E,\iota) = h.(E_{0},\iota_{0})$, we have
  \[
    \End(E,\iota) = \calO_{h.\frakc_{0},h.\lambda,\B_p}.
  \]
\end{lemma}
\begin{proof}
  See Lemma 3.3 and Lemma 7.1 of \cite{ky-pullback}.
  The result has also been described by Dorman
  \cite{dorman-orders,dorman-j}.
\end{proof}

\begin{remark}
  There are $2^{t}$ possible choices for generator of $\different{k}^{-1}\frakc_{0}/\frakc_{0}$,
  with the required norm, where $t$ is the number of prime divisors of $D$.
  However, there are only $2^{t-1}$ inequivalent ones.
  Here, we consider the orders to be equivalent if they are
  conjugate by an element of $k^{\times}$.
  Indeed, $\lambda$ and $-\lambda$ yield such equivalent conjugate orders.

  However, not knowing the specific $\lambda$ that corresponds to
  the chosen point $(E_{0},\iota_{0})$
  results in an ambiguity in \cref{prop:pvals-1} below, cf. \cite{dorman-j}.
  We will resolve this issue later on by taking the relative norm to the fixed field
  of all elements of order dividing 2 in the Galois group $\Gal(H/k)$.
\end{remark}

In the most general case that we consider,
the multiplicities involve representation numbers with additional congruences
that we will define now.
For a fractional ideal $\fraka$ of $k$, we let 
$\frakc_{\fraka} = \fraka\bar{\fraka}^{-1}\frakc_{0}$.
Moreover, we let $\lambda_{\fraka} = a.\lambda_{0} \in \different{k}^{-1}\frakc_{\fraka}/\frakc_{\fraka}$,
where $a \in \adeles_{k,f}^{\times}$ is an idele determining $\fraka$ and $\lambda_{0}$
is given in \cref{lem:quatord}. 
Note that $a$ is only unique up to an element of $\ODhatt$
but $\lambda_{\fraka}$ is a well defined element of $\different{k}^{-1}\frakc_{\fraka}/\frakc_{\fraka}$
since $\ODhatt$ acts trivially on $\different{k}^{-1}\frakc_{\fraka}/\frakc_{\fraka}$.

For $n \in \Q_{>0}$ and $\mu \in \different{k}^{-1}\fraka/\fraka$, we let
\begin{equation}
  \label{eq:rho0}
    \rho_0(n, \fraka, \mu)
  =  \#\{ x \in \frakc_{\fraka}^{-1}\fraka = \frakc_{0}^{-1}\bar{\fraka}\ \mid\ \N(x) = n, \lambda_{\fraka} x + \mu \in \fraka \}.
\end{equation}

We need to define one more quantity to describe the multiplicities $\calZ(m,\fraka,\mu)_{\frakP}$.

Let $p$ be a prime which is non-split in $k$ and define
\begin{equation}
  \label{eq:nudef}
  \nu_{p}(m) =
      \begin{cases}
            \frac{1}{2}(\ord_{p}(m)+1), & \text{if $p$ is inert in $k$},\\
            \ord_{p}(m \abs{D}), & \text{if $p$ is ramified in $k$}.
      \end{cases}
\end{equation}

Note that the prime ideals $\frakP \mid p$ of $H$ correspond to the irreducible
components of $\Spec \calO_H(\bar\F_p)$.
We let $\frakP_{0}$ be the prime ideal such that
$\pr_{\F_p}(E_{0},\iota_{0})$ with $(E_{0},\iota_{0})$ as in \cref{lem:quatord}
lies in the irreducible component corresponding to $\frakP_{0}$.
\begin{proposition}
  \label{prop:pvals-1}
  Suppose that $\Diff(m) = \{ p \}$.
  \begin{enumerate}
    \item We have $\calZ(m,\fraka,\mu)_{\frakP} = 0$ unless $m + Q(\mu) \in \Z$.
    \item For $m + Q(\mu) \in \Z$, we have
      \[
        \calZ(m,\fraka,\mu)_{\frakP_0} =
        \frac{\nu_{p}(m)}{w_k}
        \rho_0\left(\frac{m}{\kappa_{p}}\N(\fraka), \fraka, \mu \right).
      \]
  \end{enumerate}
\end{proposition}

\begin{proof}
  First note that our cycles correspond to those studied
  in \cite{ky-pullback} which are generalizations of those in \cite{kryderivfaltings}.
  The cycle $\calZ(m,\fraka,\mu)$ corresponds to $\calZ(m\abs{D};\different{k}\fraka^{-1},\lambda',\lambda'\mu)$
  for a generator $\lambda' \in \fraka^{-1}/\fraka^{-1} \different{k}$.
  
  The push-forward $\pr_{\ast} [\calZ(m)]$ is given by a formal sum
  \[
    \sum_{\frakP \subset \calO_{H}} n_{\frakP} \frakP.
  \]
  We will now determine the multiplicities.
  Fix a rational prime $p$ and a prime ideal $\frakP \subset \calO_H$ over $p$.
  Moreover, fix any geometric point $\xi = (E_{0},\iota_{0}) \in C_{D}(\overline{\kappa(\frakP)})$.
  Using Propositon 4.1 in \cite{ky-pullback},
  we see that the length $\lg \hat \calO_{\calZ(m),\xi}$ of the completed local ring
  is given by
  \begin{equation}
    \label{eq:9}
    \lg \hat \calO_{\calZ(m),\xi} = \nu_p(m).
  \end{equation}
  Note that in the notation of \cite{ky-pullback}, we have $\partial=\different{k} = \partial_{\lambda}$
  and $\Delta = D$.
  Moreover, $\ord_{p}(m) = \ord_{p}(m\abs{D})$ for $p \nmid D$.
  Therefore,
  \[
    n_{\frakP} = \frac{\nu_{p}(m)}{w_{k}} \cdot \# \{ x \in L(\xi) \mid (\xi,x) \in \calZ(m)(\bar{\F}_{p}) \}.
  \]
  
  Thus, what is left is to count the number of endomorphisms $x$, such that
  $(E_{0},\iota_{0},x) = (\xi,x) \in \calZ(m)(\overline{\kappa(\frakP)})$. 
  That is, we need to count the number of
  $x \in L(E_{0},\iota_{0})\different{k}^{-1}\fraka$, such that
    \[
      \N(x) = m \N (\fraka), \quad x + \mu \in \calO_{E_0}\fraka.
    \]
  
  The endomorphism ring $\calO_{E_0} = \End(E_{0})$ is a maximal order contained in
  the quaternion algebra $\B_p = k \oplus k y_{0}$, where $y_{0}^{2} = \kappa_{p}$.
  By \cref{lem:quatord}, we have
  \[
    \calO_{E_{0}} = \calO_{\frakc_{0},\lambda_0,\B}
    =  \{ [\gamma, \delta] \ \mid\ \gamma \in \different{k}^{-1},
       \ \delta \in \frakc_{0}^{-1},\ \gamma + \lambda_0 \delta \in \OD\}.
  \]
  This implies that
  \[
    L(E_{0},\iota_{0}) = \calO_{0} \cap k y_{0} = \frakc_{0}^{-1}\different{k} y_{0}.
  \]
  
  An element $x \in L(E_0,\iota_{0})\different{k}^{-1}\fraka$ is therefore
  of the form $x = \alpha y_{0}$ for
  \[
   \alpha \in \frakc_{0}^{-1}\overline{\fraka}.
  \]
  By Proposition 7.1 of \cite{ky-pullback}, we have that
  \[
    \fraka \calO_{\frakc_{\fraka},\lambda_{\fraka},\B_p} = a \calO_{\frakc_{\fraka},\lambda_{\fraka},\B_p} = \calO_{\frakc_{0},\lambda_{0},\B_{p}} a
    = \calO_{\frakc_{0},\lambda_{0},\B_{p}}\fraka,
  \]
  where $a \in \adeles_{k,f}^{\times}$ is an idele determining $\fraka$.
  Note that this does not depend on the choice of such an $a$
  because the order is invariant under the action
  of $\ODhatt$.
  
  Consequently, the condition $\mu + \alpha y_{0} \in \calO_{E_{0}}\fraka$ is equivalent to
  $\mu \in \different{k}^{-1}\fraka$ and $\alpha \in \frakc_{0}^{-1}\bar{\fraka}$ such that
  $\mu + \lambda_{\fraka}\alpha \in \fraka$. The norm of $\alpha$ is required to be
  $\N(\alpha) = (m/\kappa_{p})\N(\fraka)$.
  This yields the representation number $\rho_0((m/\kappa_{p})\N(\fraka), \fraka, \mu )$
  and ends the proof.
\end{proof}

We can avoid the ambiguity in the formulas above by taking the quotient $\Clk/\Clk[2]$
by the subgroup $\Clk[2]$ of elements of order dividing 2.
This corresponds to calculating the valuation at primes
$\ell \subset \calO_{L}$, where $L \subset H$ is the subfield fixed by all elements of
order 2 in $\Gal(H/k)$. We obtain a $0$-cycle on $\Spec \calO_{L}$ via the projection
$\Spec \calO_{H} \rightarrow \Spec \calO_{L}$.

For an ideal class $[\frakc] \in \Clk$ and
a positive integer $n$ we define the representation number
\[
  \rho(n,[\frakc]) = \abs{\{ \frakb \subset \OD \ \mid\ \N(\frakb) = n,\, \frakb \in [\frakc] \}}.
\]
\begin{proposition}
  \label{prop:pvals-2}
  Let $L \subset H$ be the fixed field of $\Gal(H/k)[2]$, where $H$ is the Hilbert class field of $k$.
  Let $[\frakc] \in \Clk$ be an ideal class and let $\sigma$
  correspond to $[\frakc]$ under the Artin map.
  Moreover, let $\frakf \subset \calO_{L}$ be the prime ideal below $\frakP_{0}$.
  We have for $m + Q(\mu) \in \Z$ that
    \[
      \calZ(m,\fraka,\mu)_{\frakf^{\sigma}} =
      2^{o(m)-1} \nu_{p}(m) \rho(m\abs{D}/p, [\frakc]^{-2}[\frakc_{0}\fraka]),
    \]
    where $\nu_{p}(m)$ is given before \cref{prop:pvals-1} and $o(m)$ is the number of primes
    $p \mid D$ such that $\ord_{p}(m\abs{D})>0$.
\end{proposition}
\begin{proof}
  It is enough to consider the case $[\frakc] = [\OD]$, that is, to determine the multiplicity for
  the prime $\frakf$. The general formula follows by the action of the Galois group given in \cref{prop:actZsigma}.
  We need to calculate the sum
  \[
    f \sum_{\frakP \mid \frakf} \calZ(m,\fraka,\mu)_{\frakP} = \sum_{\tau \in \Gal(H/L)} \calZ(m,\fraka,\mu)_{\frakP_{0}^\tau},
  \]
  where $f = 2$ if $p$ is ramified in $k$ and $D$ is not a prime and $f = 1$, otherwise.
  (This is the ramification degree of $\frakP \mid \frakf$.)
  According to \cref{prop:pvals-1} and \cref{prop:actZsigma}, this is equal to
  \[
    \sum_{\tau \in \Gal(H/L)} \calZ(m,\fraka,\mu)_{\frakP_{0}^{\tau}}
    = \frac{1}{w_k}\nu_{p}(m) \sum_{\substack{h \in \calC \\ h^{2} = 1}}
                       \rho_0\left(\frac{m}{\kappa_{p}}\N(\fraka), h^{-1}.\fraka, h^{-1}.\mu \right),
  \]
  where $\calC = k^{\times} \bs \adeles_{k,f}^{\times} / \hat{\calO}_{k}^{\times} \cong \Clk$
  acts as $\GSpin_{U}(\adeles_f)$ for $U = k$ as described above.
  The elements of order less or equal to $2$ in the class group $\Clk$
  correspond to the prime divisors of $D$.
  If $h^{2} = 1$, then $h.\fraka = \fraka$ because $\frakp/\bar{\frakp}=\OD$ for prime divisors of $\different{k}$.
  As $h$ ranges over $\calC[2]$, $h.\mu$ runs through 
  a set of representatives of all $\beta \in \different{k}^{-1}\fraka/\fraka$ with $\N(\beta) \equiv \N(\mu) \bmod{\N(\fraka)}$
  modulo the action of $\pm 1$. Each of these $\beta$ is counted with multiplicity
  \[
  \begin{cases}
    2^{o(m)-1} &\text{if } o(m) \geq 1,\\
    1 &\text{otherwise}.
  \end{cases}
  \]

  Finally, if $\alpha \in \frakc_{0}^{-1}\overline{\fraka}$ with $\N(\alpha) = (m/\kappa_{p})\N(\fraka)$,
  then $\tilde\fraka = \alpha \frakc_{0}\overline{\fraka}^{-1} \subset \OD$ is an integral ideal
  with
  \[
   \N(\tilde\fraka) =\frac{m}{\kappa_{p}} \cdot \frac{\kappa_{p}\abs{D}}{p} = \frac{m \abs{D}}{p}
  \]
  which lies in the class $[\tilde\fraka] = [\frakc_{0}\fraka]$.
  In this correspondence $\alpha \mapsto \tilde\fraka$, each ideal occurs with multiplicity
  \[
  \frac{w_k}{2} \cdot
  \begin{cases}
    2 &\text{if } o(m) \geq 1,\\
    1 &\text{otherwise},
  \end{cases}
  \]
  because $-\mu$ is in the set $\{h.\mu\ \mid\ h^2 = 1\}$ if and only if $o(m) = o(\mu) \geq 1$.
\end{proof}

\begin{proposition}
  \label{prop:Zprime}
  Let $D=-l$ for a prime $l \equiv 3 \bmod{4}$.
  Let $m \in \Q$ and assume that $\Diff(m) = \{p\}$ and $m+Q(\mu) \in \Z$.
  Fix an embedding of $H = k(j)$ into $\C$.

  There is a unique prime ideal $\frakP \mid p$ of $H$ fixed by complex conjugation,
  $\bar\frakP = \frakP$ and we have
  \[
    \calZ(m,\fraka,\mu)_{\frakP} =
      2^{o(m)-1} \nu_{p}(m) \rho(m\abs{D}/p, [\fraka]).
  \]
\end{proposition}
\begin{proof}
  First note that the class number of $h_{k}$ is odd \cite{zagier-qf}.
  Since $p$ is non-split in $k$, the unique prime $\frakp \subset \OD$ above $p$ splits completely in $H$.
  Therefore, the number of primes of $\calO_{H}$ above $p$
  is odd and there is at least one prime fixed by complex conjugation.
  Let $\frakP$ be such a prime. 
  Let $\tau$ denote complex conjugation $x \mapsto \bar{x}$.
  Since $\sigma \circ \tau = \tau \circ \sigma^{-1}$ for all $\sigma \in \Gal(H/k)$, we have
  \[
    \frakP^{\sigma(\frakb^{-1})} = \overline{\frakP^{\sigma(\frakb)}}
  \]
  for every $\frakb$.
  Suppose that $\frakQ$ is another prime above $\frakp$ with $\frakQ=\bar\frakQ$
  and $\frakQ = \frakP^{\sigma(\frakb)}$.
  Then it is easy to see that $\frakP^{\sigma(\frakb^{2})} = \frakP$.
  Thus, since $\Gal(H/k)$ acts transitively on the set of primes
  above $\frakp$ and $\frakp$ is totally split in $H$, we have that $\sigma(\frakb^2)$
  is the identity and thus $[\frakb]^2 = [\OD]$.
  Since $h_k$ is odd, this implies that $[\frakb] = [\OD]$
  and therefore $\frakP$ is the only prime fixed by complex conjugation.

  Let $W$ be the completion of the maximal unramified extension
  of $\calO_{H,\frakP}$ (here, $\calO_{H,\frakP}$ is
  the completion of $\calO_{H}$ with respect to $\frakP$).
  Fix an algebraic closure $\overline{\kappa(\frakp)}$ of the residue field
  $\kappa(\frakp) = \OD/\frakp$.
  The ring $W$ is a complete discrete valuation ring with maximal ideal $\pi$
  and its residue field $W/\pi$ is algebraically closed
  and therefore isomorphic to $\overline{\kappa(\frakP)} \cong \overline{\kappa(\frakp)}$
  (see Corollary 1 of Chapter II in \cite{serre-local-fields}).
  Recall the diagram in the proof of \cref{prop:compActions}.
  We can consider a similar diagram with $W$ in place of $\C_{\frakp}$.
  The bijection $C_D^{+}(\C) \rightarrow C_D^{+}(W)$ is obtained by mapping a CM elliptic curve $(E,\iota) \in \C_{D}^{+}(\C)$
  to an elliptic curve $(\tilde{E},\iota)$ over $W$ with $j$-invariant $j(\tilde{E}) = j(E)$.
  Such an elliptic curve with good reduction exists by the theorem of
  Serre and Tate \cite{serre-tate, grosszagier-singularmoduli}
  and is unique up to $W$-isomorphism.
  That all the maps involved are bijections
  is a consequence of the canonical lifting theorem \cite{howard-barbados-notes, Lan}.
  
  Now let $E$ be an elliptic curve over $W$ with $j$-invariant $j(E) = j(E_{\OD})$.
  Using the description above, we see that the reduction of
  $E$ maps to the homomorphism $\calO_H \to \overline{\kappa(\frakp)}$
  such that the image of $j(E_{\OD})$ is contained in $\F_{p}$, which corresponds to $\frakP$.
  
  As in Lemma 3.5 of \cite{grosszagier-singularmoduli}, we have that $\End_{W/\pi}(E)$
  is isomorphic to $\calO_{\OD,\lambda,\B_p}$, where $\lambda$ is any of the two possible
  $\lambda \in \OD/\different{k}$ with $\N(\lambda) \equiv -p \bmod{\abs{D}}$.
  Therefore, $\frakP = \frakP_0^{\sigma}$, where $\sigma = \sigma(\frakb)$ with
  $[\frakb]^{2} = [\frakc_0] = [\frakp_0]$.

  Thus, we obtain according to \cref{prop:pvals-2} that
  \[
    \calZ(m,\fraka,\mu)_{\frakP} = \calZ(m,\fraka,\mu)_{\frakP_0^{\sigma}} =
      2^{o(m)-1} \nu_{p}(m) \rho(m\abs{D}/p, [\fraka]).\qedhere
  \]
\end{proof}

We can now also give a formula for the \emph{Arakelov degree} $\widehat\deg\ \calZ(m,\fraka,\mu)$.

Following \cite{ky-pullback}, we define
\begin{align*}
 \widehat\deg\ \calZ(m,\fraka,\mu)
                         &= \sum_{p} \log p \sum_{x \in \calZ(m,\fraka,\mu)(\bar\F_{p})} \frac{1}{\abs{\Aut_{C_{D}}(\varphi(x))}} \lg(x)\\
                         &= \frac{1}{w_k} \sum_{p} \log(p) \sum_{x \in \calZ(m,\fraka,\mu)(\bar\F_{p})} \lg(x)
\end{align*}
and the sum runs over all rational primes.
Here, we define
\[
  \lg(x) = \text{length of } \calO_{\calZ(m,\fraka,\mu),x} = \text{length of } \hat\calO_{\calZ(m,\fraka,\mu),x}.
\]
This definition can also be expressed as $\deghat\, \calZ(m,\fraka,\mu) = \deghat \pr_{\ast}[\calZ(m,\fraka,\mu)]$,
where the latter is the usual Arakelov degree of an arithmetic divisor on the arithmetic curve given by $\Spec \calO_{H}$.
We use \cref{prop:pvals-2} to calculate this degree by using that the degree map is compatible with pushforward.
We have proved the following result, which is one of the results of \cite{kry-tiny, ky-pullback}.

\begin{corollary}
  Assume that $m + Q(\mu) \in \Z$ and $\Diff(m) = \{p\}$. Then we have
  \[
    \widehat\deg\ \calZ(m,\fraka,\mu) = 2^{o(m)-1}(\ord_p(m)+1)\rho(m \abs{D}/p, [[\frakc_0\fraka]])\log(p),
  \]
  where $\rho(n, [[\frakb]])$ is the number of integral ideals of $\OD$
  of norm $n$ in the genus of $\frakb$.
\end{corollary}

\section{An integral model for the modular curve}
\label{sec:an-integral-model}
We recall some of the properties of the integral model for the modular curve $Y_{0}(N)$
and its compactification $X_0(N)$. These models have
been intensively studied by Deligne, Rapoport \cite{deligne-rapoport},
Katz and Mazur \cite{katz-mazur}. We refer to these references and \cite{grosszagier}
for details.

The stack $\calY_{0}(N)$ ($\calX_0(N)$) over $\Z$ represents the moduli problem
that assigns to any base scheme $S$ the cyclic isogenies of degree $N$ of (generalized)
elliptic curves $\pi: E \rightarrow E'$ over $S$ such that $\ker \pi$
meets every irreducible component of each geometric fiber.
On complex points, we have $\calY_0(N)(\C) = Y_0(N)(\C)$ and  $\calX_0(N)(\C) = X_0(N)(\C)$.

Here, the condition that $A = \ker \pi$ is cyclic of degree $N$
means that locally on $S$ there is a point $P$ such that
\[
  A = \sum_{a=1}^{N}[aP]
\]
as a Cartier divisor on $E$.
This becomes the usual condition that $A$ is locally isomorphic
to $\Z/N\Z$, when $N$ is invertible in $S$. We will always assume that
$N$ is square-free. In this case the condition means that $A$ is locally free of rank $N$.

The cusps correspond to certain degenerated elliptic curves \cite{deligne-rapoport}.
We will not give a precise definition of these as we will mostly work
on the substack $\calY_0(N)$.
\begin{theorem}[Theorems 1.2.1 and 3.2.7 of \cite{conrad-x0n}]
  Let $N$ be square-free.
  Then the stack $\calX_0(N)$ is a proper flat Deligne-Mumford stack over $\Z$.
  It is regular and has geometrically connected fibers of pure dimension one.
  Moreover, the stack $\calX_0(N)$ is smooth over $\Z[1/N]$.
\end{theorem}

\subsection{Integral extensions of Heegner divisors}
\label{sec:integr-extens-heegn}
The following moduli problem describes a natural extension of the divisor
$Z(d,r)$ defined in the introduction to the stack $\calX_0(N)$.
We follow \cite{bryfaltings}, Section 7.3.
\begin{definition}
  \label{def:Zmu-stack}
  Let $d \in \Z_{<0}$ and $r \in \Z$ such that $d \equiv r^2 \bmod{4N}$.
  The integer $d$ is a negative discriminant and
  we denote by $\calO_d$ the order of $k = \Q(\sqrt{d})$
  of discriminant $d$. The ideal $\frakn = (N, \frac{r+\sqrt{d}}{2})$
  has norm $N$.
  We define $\calZ(d,r)$ to be the Deligne-Mumford stack representing
  the moduli problem which assigns to a base scheme $S$ over $\Z$
  the set of pairs $(\pi: E \rightarrow E', \iota)$, such that
\begin{enumerate}
\item $\pi: E \rightarrow E'$ is a cyclic degree $N$ isogeny of two elliptic curves $E$ and $E'$ over $S$,
\item $\iota: \calO_d \hookrightarrow \End(\pi) = \{\alpha \in \End(E): \pi\alpha\pi^{-1}\in \End(E')\}$
      is an $\calO_d$-action on $\pi$ such that $\iota(\frakn)\ker \pi = 0$.
\end{enumerate}
\end{definition}
There is a natural morphism
\[
  \calZ(d,r) \longrightarrow \calX_0(N),
\]
given by the forgetful map $(\pi: E \rightarrow E',\iota) \mapsto \pi: E \rightarrow E'$.
Note that $\calZ(d,r)$ does not intersect the boundary $\calX_{0}(N) \bs \calY_{0}(N)$ \cite{conrad-gz}.
We recall the following two facts from \cite{bryfaltings}.
\begin{lemma}
  The forgetful map $\calZ(d,r) \longrightarrow \calX_0(N)$ is finite and \'{e}tale.
  The stack $\calZ(d,r)$ defines a horizontal divisor on $\calX_0(N)$.
\end{lemma}
\begin{lemma}
  As divisors in the complex fiber, we have
  \[
    \calZ(d,r)(\C) = Z(d,r).
  \]
  The divisor $\calZ(d,r)$ is in fact the flat closure of $Z(d,r)$.
\end{lemma}
We briefly mention the relation to the Heegner points as defined by
Birch \cite{Birch-heegner1}, Gross \cite{gross-heegner-x0n} and Gross-Zagier \cite{grosszagier}.
A Heegner point on $X_{0}(N)(\C)$ is described by the data $(\calO, \frakn, [\fraka])$,
where $\calO \subset k$ is an order, $\frakn \subset \calO$ is a proper $\calO$-ideal
with quotient $\calO/\frakn$ cyclic of order $N$ and $[\fraka]$ is the class of some
invertible $\calO$-module $\fraka$ in $\Pic(\calO)$. The Heegner point corresponding to this data
is given by the diagram
\[
  \C/\fraka \rightarrow \C/\fraka\frakn^{-1}.
\]
If we choose an oriented basis $(\omega_1,\omega_2)$ of $\fraka$, such that
$\fraka\frakn^{-1} = (\omega_1,\omega_2/N)$, then
the point in $X_0(N)(\C) \cong \Gamma_0(N) \bs \uhp$ is given by the orbit of
$\tau = \omega_1/\omega_2$ \cite{gross-heegner-x0n}.

\section{CM values of modular functions}
\label{sec:cm-modular-functions}
The starting point for our study of CM values of modular functions
with zeros and poles supported on Heegner divisors is the following Lemma.
\begin{lemma}[\cite{bryfaltings}, Lemma 7.10]
  \label{lem:cdZisom}
  Let $D$ be a negative \emph{fundamental discriminant} and assume that $D \equiv 1 \bmod{4}$.
  Let $r \in \Z$ such that $D \equiv r^2 \bmod{4N}$.
  There is an isomorphism of stacks
  \begin{equation*}
    \jD: C_{D} \rightarrow \calZ(D,r), \quad (E,\iota) \mapsto (\pi: E \rightarrow E/E[\frakn], \iota).
  \end{equation*}
\end{lemma}
Here, we denote by $E[\frakn]$ the kernel of multiplication by elements in $\frakn$.
Combining the map $\jD$ with the forgetful map $\calZ(D,r) \rightarrow \calX_{0}(N)$ yields
a map $C_{D} \rightarrow \calX_{0}(N)$, still denoted $\jD$.
Note that this map also depends on the choice of $r$. For simplicity, we do not reflect this in the notation.

Fix $D_0, D_1 \in \Z_{<0}$ and assume thatthey satisfy the properties of \cref{def:Zmu-stack}.
In particular, there are $r_{0}, r_{1} \in \Z$, such that $r_i^2 \equiv D_{i} \bmod{4N}$.
We write $\frakn_{i}$ for the corresponding ideals of norm $N$ in $\calO_{D_{i}}$
generated by $N$ and $\frac{r_{i}+\sqrt{D_{i}}}{2}$.
Moreover, assume that $D_0$ is a fundamental discriminant with $D_0 \equiv 1 \bmod{4}$ and
$D_0D_{1}$ is not a perfect square so that $\calZ(D_0, r_0)$ and
$\calZ(D_1,r_1)$ intersect properly.

For a CM elliptic curve $(E,\iota) \in C_{D_{0}}(S)$, we define
\[
  \calO_{E,\frakn_{0}} :=
  \End_{S}( E \rightarrow E/E[\frakn_{0}] ) = \{ \alpha \in \End_{S}(E)\ \mid\ \pi \alpha \pi^{-1} \in \End_{S}(E/E[\frakn_{0}]) \}.
\]
We are interested in the intersection of $\calZ(D_0,r_0)$
and $\calZ(D_1,r_1)$ on $\calX_0(N)$ or, equivalently, in the pullback of $Z(D_1,r_1)$ under $\jDz$.
The stack $\jDz^{*}\calZ(D_{1},r_{1})$ represents the following moduli problem.
For a base scheme $S$, consider the category $\calM(D_{1},r_{1}, \frakn_{0})(S)$ of triples $(E,\iota,\phi)$, where
\begin{enumerate}
\item $(E,\iota) \in C_{D_{0}}(S)$,
\item $\phi:\calO_{D_{1}} \hookrightarrow \calO_{E,\frakn_{0}}$ is an action of $\calO_{D_{1}}$, such that
\item $\phi(\mathfrak{n}_{1})E[\frakn_{0}] = 0$.
\end{enumerate}
  We consider the fiber product diagram
  \[
    \xymatrix@1{
        \jDz^{*}\calZ(D_{1},r_{1}) = \calZ(D_{1},r_{1}) \times_{\calX_{0}(N)} C_{D_{0}}  \ar[r]^-{\pi_{2}} \ar[d]^{\pi_{1}} & C_{D_{0}} \ar[d] \\
        \calZ(D_{1},r_{1}) \ar[r] & \calX_{0}(N).
    }
  \]
\begin{lemma}
  The map
  \[
    \varphi: \jDz^{*}\calZ(D_{1},r_{1}) \rightarrow \calM(D_{1}, r_{1}, \frakn_0),
  \]
  given by
  \[
    \xi \mapsto (E, \iota, \phi),
  \]
  where $\pi_{2}(\xi) = (E,\iota)$ and $\pi_{1}(\xi) = (E \rightarrow E/E[\frakn_0], \phi)$
  is an isomorphism of stacks.
\end{lemma}
\begin{proof}
  This can be found in \cite[Section 7.3]{bryfaltings}.
  It is clear that the described map is well defined and injective over any scheme $S$.
  In the other direction, suppose that $(E,\iota,\phi) \in \calM(D_{1},r_{1},\frakn_0)(S)$. Then
  $(E,\iota) \in C_{D_0}(S)$ and $(E \rightarrow E/E[\frakn_0],\phi) \in \calZ(D_{1},r_{1})(S)$ by definition.
  Thus, we obtain maps $\calM(D_{1},r_{1},\frakn_{0}) \rightarrow \calZ(D_{1},r_{1})$
  and $\calM(D_{1},r_{1},\frakn_{0}) \rightarrow C_{D_{0}}$.
  By the universal property of the fiber product, we obtain a unique
  map $\tilde{\varphi}$ that makes the following diagram commutative.
  \[
    \xymatrix@1{
        \calM(D_{1},r_{1},\frakn_0) \ar@/_/[ddr] \ar[dr]^{\tilde\varphi} \ar@/^/[drr] \\
        & \calZ(D_{1},r_{1}) \times_{\calX_{0}(N)} C_{D_{0}}  \ar[r]_-{\pi_{2}} \ar[d]^{\pi_{1}} & C_{D_{0}} \ar[d] \\
        & \calZ(D_{1},r_{1}) \ar[r] & \calX_{0}(N)
    }
  \]
  Therefore, we have $\xi = \tilde\varphi((E,\iota,\phi)) \in \jDz^{*}\calZ(D_{1},r_{1})$
  with $\varphi(\xi) = (E,\iota,\phi)$ by the definition of $\varphi$.
\end{proof}

\begin{lemma}
  \label{lem:pullback1}
  We have the identity
  \[
    \jDz^{*} \calZ(D_{1},r_{1})
        = \sum_{\substack{n \equiv r_{0} r_{1} \bmod{2N} \\ n^{2} \leq D_0D_1}}
        \calZ\left(\frac{D_0D_1-n^{2}}{4N\abs{D_{0}}}, \frakn_{0}, \frac{n+r_1\sqrt{D_{0}}}{2\sqrt{D_{0}}} \right)
  \]
  of divisors on $C_{D_{0}}$.
\end{lemma}
\begin{proof}
  This is a direct consequence of Lemma 7.12 in \cite{bryfaltings}
  where the authors prove an identity on geometric points over algebraically closed fields
  in any characteristic. We also refer to \cite{ehlen-diss} for details.
\end{proof}

To ease the notation a bit, we continue to write $D = D_0$ and let $\rho \in \Z$, such that $\rho^2 \equiv D \bmod{4N}$ (i.e. $\rho = r_0$).
For a rational function $f \in \Q(\calX_0(N))$ we consider the pullback $\jD^{*}(f \mid_{\jD(C_{D})})$,
which makes sense as long as $\jD(C_D)$ is not contained in the divisor of $f$.
The element $\jD^{*}f$ then defines an element of the function field $\Q(C_{D}) \cong \Q(\mathbf{C}_{D}) = H$, the Hilbert class field of $k_{D} = \Q(\sqrt{D})$.

Now let $f \in \Q(\calX_{0}(N))$ be a modular function such that its divisor is a linear combination of the Heegner divisors
$\calZ(m,\mu)$. That is, there are integers $c(m,r)$, such that
\[
  \divisor(f) = \sum_{r \bmod{2N}} \sum_{\substack{d \in \Q_{<0} \\ d \equiv r^2 \bmod{4N}}} c(d,r) \calZ(d,r) + C(f),
\]
where $C(f)$ is supported at the boundary.
We have by definition (cf. the proof of Proposition 3.7 in \cite{vistoli-intersection})
that
\[
  \divisor(\jD^{*}f) = \jD^{*}(\divisor(f)).
\]

Finally, we normalize the map $\pr: C_{D} \rightarrow \Spec \calO_{H}$ in the following way
(it is only unique up to an automorphism of $\calO_{H}$).
We fix an embedding of $H$ into $\C$ and the integral ideal
\[
  \frakn = \left(N,\frac{\rho+\sqrt{D}}{2}\right) \subset \calO_{D}
\]
of norm $N$.
Consider the Heegner point $z_{D,\rho}$ given by $\pi : \C/\calO_{D} \rightarrow \C/\frakn^{-1}$.
Then we require that $\pr$ is chosen such that $\pr_{\ast}\j^{*}(f^{w_{k}}) \in H$ is equal to $f(z_{D,\rho})$, where
$w_{k}$ is the number of roots of unity in $k = k_{D}$.
As a point in $\uhp$, we can take $z_{D,\rho}$, such that its image under the Fricke involution is
\[
  \frac{-1}{Nz_{D,\rho}} = \frac{-\rho+\sqrt{D}}{2N} \in \uhp,
\]
according to our remark at the end of the last section.
Note that $f(z_{D,\rho})  \in H$ defines a divisor on $\Spec \calO_{H}$ given by
\begin{equation}
  \label{eq:divf1}
  \sum_{\mathfrak{P} \subset \calO_{H}} \ord_{\mathfrak{P}}(f(z_{D,\rho})) \mathfrak{P}. 
\end{equation}
Here, the sum is over all nonzero prime ideals of $\calO_{H}$.

\begin{theorem}
  \label{thm:divf2}
  Let $D < 0$ be an odd fundamental discriminant, $\rho \in \Z$ with $\rho^2 \equiv D \bmod{4N}$ and normalize $\pr$ as described above.
  Suppose that $f \in \Q(\calX_0(N))$ with
  \[
    \divisor(f) = \sum_{r \bmod{2N}} \sum_{\substack{d \in \Z_{<0} \\ d\, \equiv\, r^2 \bmod{4N}}} c(d,r) \calZ(d,r) + C(f),
  \]
  where $C(f)$ is supported at the cusps 
  and $\divisor(f)$ and $\calZ(D,\rho)$ intersect properly. Then we have
  \[
    \ord_{\mathfrak{P}}(f(z_{D,\rho})) = w_{k} \sum_{r \bmod{2N}} \sum_{d\, \in\, \Z_{<0}} c(d,r)
        \sum_{\substack{n\, \equiv\, \rho \cdot r \bmod{2N} \\ n^{2}\, \leq \, dD}}
        \calZ\left(\frac{dD-n^{2}}{4N\abs{D}}, \frakn, \frac{n+r\sqrt{D}}{2\sqrt{D}} \right)_{\frakP}
  \]
  for every prime $\frakP$ of the Hilbert class field $H$.
\end{theorem}

\begin{proof}
  This follows directly from our considerations above
  and the fact that the pullback of $\divisor(f)$ as a Cartier divisor
  (Equation \eqref{eq:divf1}) agrees with the pullback as
  a Weil divisor, corrected by the multiplicity as explained above.
  The pullback as a Weil divisor is described in \cref{lem:pullback1}.
  Also note that the image of $C_{D}$ does not intersect the boundary.
\end{proof}

\begin{remark}
  The formulas in \cref{prop:pvals-1} and \cref{prop:pvals-2}
  provide explicit formulas for the quantities in \cref{thm:divf2}.
\end{remark}

We now recover
the Theorem of Gross and Zagier on singular moduli \cite{grosszagier-singularmoduli}
and its generalization by Dorman \cite{dorman-j}.
Let $D$ be a negative fundamental discriminant and suppose that $D$ is odd.
Moreover, let $d$ be a negative discriminant, coprime to $D$.
As in the Introduction, consider the modular function
\[
  \Psi(z,D) = \prod_{Q \in \SL_2(\Z) \bs \calQ_{D}} (j(z) - j(\alpha_{Q}))^{1/w_{k}}.
\]
Here, $\calQ_{D}$ is the set of quadratic forms of discriminant $D$ and $w_{k}$ is the number of roots of unity in $k = \Q(\sqrt{D})$

\begin{theorem}[Gross-Zagier, Dorman]
  Let $H$ be the Hilbert class field of $k = \Q(\sqrt{D})$.
  Moreover, let $L$ be the fixed field of $\Gal(H/k)[2]$, let $p$ be a rational prime
  and $\frakf \subset \calO_L$ be the prime ideal below $\frakP_0 \mid p$
  (see \cref{sec:spec-endom-p}). Then we have
  \[
    \ord_{\frakf^{\sigma}}(\Psi(z_{D,\rho},d)) = \frac{w_{k}}{4}
              \sum_{\substack{n \in \Z \\ n \equiv 1 \bmod{2}}} 2^{o(n)} \bar{\nu}_p\left(\frac{Dd - n^2}{4\abs{D}}\right)
                               \rho\left( \frac{Dd-n^2}{4p}, [\frakc]^{-2}[\frakc_0] \right),
  \]
  for $\sigma = \sigma(\frakc)$. Here, $\bar\nu_p(x) = \nu_p(x)$ if $\Diff(x)=\{p\}$ and $\bar\nu_p(x) = 0$, otherwise.
\end{theorem}

To conclude this section and justify the simpler
version (\cref{thm:divf2-intro}) of \cref{thm:divf2} in the Introduction, we show that there is a rather simple criterion to decide if the divisor of a modular function on $\calX_{0}(N)$ is horizontal if $N$ is square-free. It is reflected in the classical Kronecker congruence modulo $p$ of the modular equation of level $N$.

Recall that the fiber of $\calX_0(N)$ above $p \nmid N$ is smooth and irreducible.
We write the valuation on the function field that is induced by such a fiber by $\nu_p$.
For $p \mid N$, the fiber has two irreducible components that intersect at each
supersingular point. The cuspidal sections $\infty$ and $0$ always only intersect
one of those two fibers. We denote by $\nu_{p,\infty}$ the valuation on the function field
of $\calX_0(N)$ induced by the component above $p$ that intersects the cusp $\infty$ and by $\nu_{p,0}$
the valuation corresponding to the other component.

For a modular function $f \in \Q(X_0(N)) = \Q(j,j_N)$, where $j_N(\tau) = j(N\tau)$,
denote by
\[
  f_{\infty}(\tau) = \sum_{n\gg -\infty} c_{\infty}(n)q^{n}
\]
and
\[
  f_{0}(\tau) = f_{\infty}\left(\frac{-1}{\tau}\right) = \sum_{n \gg-\infty} c_{0}(n)q^{n/N}
\]
the Fourier expansions of $f$ at the cusps $\infty$ and $0$, respectively.

\begin{proposition}
  \label{prop:fmultFourier}
  Let $f \in \Q(j,j_N)$ be a modular function for $\Gamma_0(N)$ and assume that $N$ is square-free.
  Let $p$ be a prime and set $a = \inf\{\ord_{p}c_{\infty}(n)\}$ and $b = \inf\{\ord_{p}c_0(n)\}$.
  \begin{enumerate}
  \item If $p \nmid N$, then $a = \nu_{p}(f)$.
  \item If $p \mid N$ then $a = \nu_{p,\infty}(f)$ and $b = \nu_{p,0}(f)$.
  \end{enumerate}
\end{proposition}
\begin{proof}
  Evaluation at the Tate curve (cf. \cite{deligne-rapoport, diamond-im})
  \[
    \bar\calG_{m}^{q} / q^{\Z}
  \]
  over $\Z[[q]]$ gives a homomorphism
  \[
    \tau_{\infty}: \Spec(\Z[[q]]) \rightarrow \calX_0(N),
  \]
  and evaluation at the Tate curve
    \[
    \bar\calG_{m}^{q^{1/N}} / q^{\Z}
  \]
  over $\Z[[q^{1/N}]]$ accordingly
  \[
    \tau_{0}: \Spec(\Z[[q^{1/N}]]) \rightarrow \calX_0(N).
  \]
  Combining these maps with the geometric points
  given by $q \mapsto 0$ for $\tau_{\infty}$ and $q^{1/N} \mapsto 0$ for $\tau_0$,
  we obtain the cuspidal sections $\Spec(\Z) \rightarrow \calX_0(N)$ corresponding to
  the cusps $\infty$ and $0$, respectively.
  The pullbacks $\tau_{\infty}^{\ast}f$ and $\tau_{0}^{\ast}f$
  of $f \in \Q(\calX_0(N))$ are given by the Fourier expansion of $f$
  at $\infty$ and $0$, respectively.
  The valuation at $p$ on $\Q((q))$ and $\Q((q^{1/N}))$ is given by
  the canonical extension of the $p$-adic valuation on $\Q$.
  It agrees with the valuations given by $a$ and $b$ in the statement of the Proposition.
  If $p \nmid N$ they coincide since the fiber of $\calX_0(N)$ above $p$ is irreducible
  in this case.
  
  We refer to Theorem VI. 3.10 on page 163
  and Corollary 3.12 in \cite{deligne-rapoport} for more details.
  Moreover, in Section 3.16, ibid., the case $\calX_0(p)$ for a prime $p$ is discussed in more detail.
\end{proof}

\section{Borcherds products}
\label{sec:borcherds-products}
We will finally sketch how to apply our results to Borcherds products on modular curves.

Let $N$ be a positive integer and consider the congruence
subgroup $\Gamma_0(N) \subset \SL_2(\Z)$.
The modular curve $Y_0(N) := \Gamma_0(N) \backslash \uhp$ can be obtained as an orthogonal modular variety as follows.

Consider the vector space $V: = \{x \in M_2(\Q) \, \mid \, \tr(x)=0\}$ and define the quadratic form
by $Q(x) = -N\det(x)$. The corresponding bilinear form is $(x,y) = N\tr(xy)$.
The space $(V,Q)$ has signature $(2,1)$.

The symmetric domain $\domain$ of $\SO_V(\R) \cong \SO(2,1)$ can be identified with the Grassmannian
of two-dimensional positive definite subspaces of $V(\R)$.
It is isomorphic to the complex upper half-plane $\uhp$ via
  \begin{equation*}
    z = x+iy \mapsto \left[ \zmatrix \right] \mapsto
    \R \Re \zmatrix
    \oplus \R \Im \zmatrix.
  \end{equation*}
  The action of $\gamma \in \SL_2$ is explicitly given by
  \begin{equation*}
    \gamma \cdot \zmatrix =
    (cz+d)^2 \begin{pmatrix}
      \gamma z & -(\gamma z)^2 \\ 1 & -\gamma z
    \end{pmatrix},
  \end{equation*}
  where $\gamma z$ is the action via linear fractional transformations on $\uhp$.

In $V$ we have the even lattice
\[
  L = \left\lbrace
    \begin{pmatrix}
      b & -\frac{a}{N} \\ c & -b
    \end{pmatrix} \ \mid\ a,b,c \in \Z
      \right\rbrace.
\]
The dual lattice of $L$ is given by
\[
  L' = \left\lbrace
    \begin{pmatrix}
      \frac{b}{2N} & -\frac{a}{N} \\ c & -\frac{b}{2N}
    \end{pmatrix} \ \mid\ a,b,c \in \Z
    \right\rbrace.
\]

Note that the discriminant group $L'/L$ is cyclic of order $2N$
and we can identify the corresponding finite quadratic module with
the group $\Z/2N\Z$ together with the quadratic form $x^{2}/4N$,
valued in $\frac{1}{4N}\Z/\Z \subset \Q/\Z$.
We denote by $M_{k,L}^!$ the space of vector valued weakly holomorphic modular forms
of weight $k$ and representation $\rho_L$, the Weil representation associated with $L$.

Let $f \in M_{1/2,L}^!$ with Fourier expansion
\[
  f(\tau) = \sum_{\mu \in \Z/2N\Z} \sum_{m \in \Q} c_f(m,\mu)e(m\tau)\phi_\mu,
\] 
and $c_f(m,\mu) \in \Z$ for all $m \leq 0$ and all $\mu \in \Z/2N\Z$.
Borcherds \cite{boautgra} 
showed that there is a meromorphic modular form $\Psi_L(z,f)$ of weight $c_f(0,0)$ for $\Gamma_0(N)$
with divisor
\[
\divisor(\Psi_L(z,f)) = Z(f) = \sum_{r \bmod{2N}}\sum_{m<0} c_f(m,r) Z(4Nm,r)
\]
on $Y_0(N)$.
Moreover, $\Psi_L(z,f)$ has an infinite product expansion of the form
  \[
    \Psi_L(z,f) = e((\rho_{f},z))\prod_{n=1}^{\infty} (1-e(nz))^{c_{f}(n^{2}/4N,n)},
  \]
which converges for $\Im(z)$ large enough and where $\rho_{f}$ is the corresponding Weyl vector at the cusp $\infty$.
We refer to Borcherds \cite[Theorem 13.3]{boautgra} and Bruinier and Ono \cite[Theorem 6.1]{bronolderiv} for details.

\begin{lemma}
  Let $f \in M_{1/2,L}^{!}(\Z)$ with constant coefficient $c_f(0,0) = 0$ and $c_f(m,\mu) \in \Q$ for all $m \in \Q$ and $\mu \in L'/L$.
  Then there exists an integer $M_{f}$, such that the Borcherds product
  $\Psi_{L}(z,h,M_{f} \cdot f)$ defines a meromorphic modular function contained in $\Q(j,j_{N})$.
\end{lemma}
\begin{proof}
  Since the Fourier coefficients of $f$ have bounded denominators,
  replacing $f$ by an integral multiple $f' = M \cdot f$ we obtain only integral coefficients.
  We view $\Psi_{L}(z,h,f')$ as a meromorphic function on $X_0(N)$.  
  The multiplier system of $\Psi(z,f')$ has finite order,
  which can be shown using the embedding trick (\cite[Lemma 8.1]{boautgra}, \cite{Borcherds-GKZ-Corr}).
  Together with the integrality of the $c_{f'}(m,\mu)$ this implies
  that the Fourier expansions of $\Psi(z,f')$ have rational coefficients at all cusps (for square-free $N$).
  Thus, $\Psi(z,M'f)$ for some $M' \in \Z$ is contained in the field $\Q(X_0(N)) = \Q(j,j_N)$ by the $q$-expansion principle.
\end{proof}

\begin{theorem}
  \label{thm:bopintegral}
  Let $f \in M_{1/2,L}^{!}(\Z)$ be a weakly holomorphic modular form
  with only integral Fourier coefficients and assume that $N$ is square-free.
  Suppose that the multiplier system of $\Psi_L(z,h,f)$ is trivial.
  Then the divisor of the rational function defined by $\Psi_L(z,h,f)$ on $\calY_{0}(N)$ is equal to $\calZ(f)$,
  the flat closure of $Z(f)$.
\end{theorem}

\begin{proof}
  The triviality of the multiplier system implies that $\Psi(z) := \Psi_L(z,f) \in \Q(j,j_N)$,
  as we have seen. The product expansion implies that the Fourier expansion of $\Psi(z)$
  at the cusp $\infty$ has coprime integral coefficients under our assumption on $f$.
  Therefore, the fibers for $p \nmid N$ cannot occur in the divisor.

  The Fricke involution $W_{N}$ is contained in $\GSpin_V(\Q)$
  and its image belongs to $\SO^{+}(L)$ (cf. \cite{bronolderiv}).
  We denote by $\sigma_{N}$ the image of $W_N$ in $\Og(L'/L)$.
  It acts on $L'/L$ as $\mu \mapsto -\mu$ and therefore,
  \[
    -4 \log\abs{(\Psi\mid W_{N}) (z)} = \Phi\left(\frac{-1}{Nz} , f \right) = \Phi(z, f_{N}) = -4 \log\abs{\Psi(z,f_{N})},
  \]
  where $f_{N} = f^{\sigma_{N}}$.

  We have that $c_{f}(m,\mu) = c_f(m,-\mu)$ by the action of the center of $\Mp_2(\Z)$.
  Therefore, $f_{N} = f$ and ${\Psi\mid{W_{N}}(z)} = \pm \Psi(z)$ because $W_N^2=1$.
  This implies that the Fourier expansion of $\Psi(z)$ at the cusp $0$ has also
  coprime integral Fourier coefficients and the result follows by \cref{prop:fmultFourier}.
\end{proof}

Now let $z_{D,\rho}$ be the Heegner point of level $N$ given by
\[
  z_{D,\rho} = \frac{\rho+\sqrt{D}}{2N} \in \uhp,
\]
where we assume as usual that $D<0$ is an odd fundamental discriminant.
We conclude that if $z_{D,\rho}$ is not contained in the divisor of $\Psi_L(z,f)$, then
the prime valuations of $\Psi_{L}(z_{D,\rho},f)$ can be obtained using \cref{thm:divf2}.

\section{Some numerical examples}
\label{sec:examples}
We give some examples to illustrate and test our formula numerically.

The Borcherds products discussed in \cref{sec:borcherds-products}
provide an easy way to generate examples if the space of obstructions $S_{3/2,L^-}$ is known. 
Here, $S_{3/2,L^-}$ is the space of cusp forms transforming with the Weil representation
$\rho_{L^-}$ of the lattice given by $L$ with the negative $-Q$ of the associated quadratic form.
According to \cite{bef-simple}, if $N=p$ is prime, then $S_{3/2,L^-} = \{0\}$ if and only if the modular curve 
$X_0^+(p) = \Gamma_0^+(p) \bs \uhp$ has genus zero. Here, $\Gamma_0^+(p)$ is the extension
of $\Gamma_0(p)$ by the Fricke involution.

As an explicit and non-trivial example of this kind let us take $N=47$.
Since $S_{3/2,L^-}$ vanishes, there exists a weakly holomorphic modular form $f$
of weight $1/2$ with representation $\rho_L$ and principal part 
\[
  \frac{1}{2}q^{-11/188}(\phi_{41} + \phi_{-41}).
\]
We can assume that the constant term of $f$ vanishes.
Its Borcherds lift can be identified as the Hauptmodul for $X_0^+(47)$ which has a Fourier development starting with
\[
\Psi(z,f) = q^{-1} +  1 + q + 2q^{2} + 3q^{3} + 3q^{4} + 5q^{5} + 5q^{6} + 8q^{7} + 9q^{8} + 12q^{9} + 14q^{10} +  O(q^{11}),
\]
where $q=e(z)$.
It can be explicitly constructed as the quotient
\[
\Psi(z,f) = \frac{\theta_1(z)-\theta_2(z)}{2\eta(z)\eta(47z)} +1,
\]
where $\theta_i(\tau)$ are the two theta functions corresponding to 
the binary quadratic forms $Q_1 = [1,1,12]$ and $Q_2 = [2,-1,6]$ of discriminant $-47$.
Its divisor is equal to $\frac{1}{2}(Z(11,41)+Z(11,-41))$ which equals $Z(11,41)$ on $X_0^+(47)$.
Since it is the generator of the function field of $X_0^+(47)$, its values at CM points of
fundamental discriminant generate the corresponding Hilbert class field.
As it turns out, the class polynomials for $\Psi(z,f)$ have some advantages over
the usual Hilbert class polynomials for the elliptic modular function $j(\tau)$.
In particular, their discriminants and resultants are usually much smaller.

As a first example, take the CM points of discriminant $-23$. \cref{thm:divf2} easily reveals
that $\Psi(z,f)$ is a unit in the Hilbert class field for any such point (note that it cannot be equal to $0$). 
In fact, it is easily confirmed numerically that $\Psi(z_{-23,27},f)$ is a root of the polynomial
\[
x^{3} - x^{2} + 2x - 1
\]
with discriminant $-23$.
Using \sage{}, we confirmed that these roots are indeed contained in the Hilbert class field of $\Q(\sqrt{-23})$
and any of these generates the Hilbert class field over $\Q(\sqrt{-23})$ (the polynomial is irreducible and since the constant coefficient is equal to $1$, the values are indeed units).
To compare with the $j$-invariant, note that the Hilbert class polynomial for $j$ is equal to
$x^{3} + 3491750 \, x^{2} - 5151296875 \, x + 12771880859375$ and has discriminant
$-1 \cdot 5^{18} \cdot 7^{12} \cdot 11^{4} \cdot 17^{2} \cdot 19^{2} \cdot 23$.

To see a non-unit, we can take for instance $D = -107$.
In this case, we evaluate at the point $z_{-107,9}$.
Numerically, we obtain $\Psi(z_{-107,9},f) \approx 2.796321\ldots$.
This turns out to be the unique real root of $x^{3} - 3x^{2} + 2x - 4$. 
We consult \cref{thm:divf2} together with \cref{prop:Zprime} to obtain
that all prime ideals that contribute to its prime ideal factorization lie above $2$.
Indeed, the only non-zero contribution in the Theorem comes from $n= \pm 7$
and $\Diff(6/107) = \{2\}$.

We fix the embedding of the Hilbert class field $H$ of $k=\Q(\sqrt{-107})$ into $\C$,
such that $H=k(j)$ and $j$ is mapped to $j(\calO_k) \approx -1.297832\ldots \in \R$.
Let $\frakP_0, \frakP_1$ and $\frakP_2$ be the prime ideals of $H$ above $2$
where $\frakP_0$ is fixed by complex conjugation and $\overline{\frakP_1} = \frakP_2$.
Since the two primes of $\Q(\sqrt{-107})$ above $3$ are both not principal, the valuation at $\frakP_0$ is equal to zero by
\cref{prop:Zprime}.
The valuation at the other two primes turn both out to be equal to $2\cdot(1/4+1/4) = 1$, which we can also confirm using \sage{}.

\printbibliography

\end{document}